\documentclass[a4paper,11pt]{article}

\RequirePackage{caption}

\usepackage[multidot]{grffile}
\usepackage{hyperref} 
\hypersetup{
    colorlinks=true,
    linkcolor=blue,
    filecolor=magenta,      
    urlcolor=cyan,
    citecolor=red,
    pdfpagemode=FullScreen,
    }

\usepackage[margin=1in]{geometry}
\usepackage{setspace}
\usepackage{floatrow}
\usepackage{enumitem}
\usepackage{hyperref}
\usepackage{array, theoremref}
\RequirePackage{latexsym,amsfonts,amssymb,amsbsy,amsmath,amsthm,enumerate,epsfig}
\usepackage{subfig, epsfig}
\usepackage{graphicx}
\usepackage{tikz}
\usepackage{color}

\usepackage{supertabular} 

\usepackage{latexsym}
\usepackage{amssymb, amsfonts}
\usepackage{textcomp}
\usepackage{mathrsfs}

\usepackage{amsthm, amsmath}

\usepackage[ruled]{algorithm2e}

\newtheorem{thm}{Theorem}[section]

\newtheorem{lemma}{Lemma}

\theoremstyle{definition}
\newtheorem{definition}[thm]{Definition}

\newcommand{\R}{\mathbb{R}}

\newcommand{\vx}{\mathbf{x}}

\newcommand{\vu}{\mathbf{u}}
\newcommand{\vv}{\mathbf{v}}

\newcommand{\abs}[1]{\left|#1\right|}
\newcommand{\norm}[1]{\abs{\abs{#1}}}
\newcommand{\inprod}[1]{\langle #1\rangle}

\newcommand{\be}{\begin{equation}}  
\newcommand{\ee}{\end{equation}}

\graphicspath{{./image/}}

\begin{document}

\title{Local Trajectory Variation Exponent (LTVE) for \\ Visualizing Dynamical Systems}

\author{
Yunchen Tsai\thanks{Department of Mathematics, the Hong Kong University of Science and Technology, Clear Water Bay, Hong Kong. Email: {\bf yctsai@connect.ust.hk}}
\and
Shingyu Leung\thanks{Department of Mathematics, the Hong Kong University of Science and Technology, Clear Water Bay, Hong Kong. Email: {\bf masyleung@ust.hk}}
}

\date{}

\maketitle

\begin{abstract}
The identification and visualization of Lagrangian structures in flows plays a crucial role in the study of dynamic systems and fluid dynamics. The Finite Time Lyapunov Exponent (FTLE) has been widely used for this purpose; however, it only approximates the flow by considering the positions of particles at the initial and final times, ignoring the actual trajectory of the particle. To overcome this limitation, we propose a novel quantity that extends and generalizes the FTLE by incorporating trajectory metrics as a measure of similarity between trajectories. Our proposed method utilizes trajectory metrics to quantify the distance between trajectories, providing a more robust and accurate measure of the LCS. By incorporating trajectory metrics, we can capture the actual path of the particle and account for its behavior over time, resulting in a more comprehensive analysis of the flow. Our approach extends the traditional FTLE approach to include trajectory metrics as a means of capturing the complexity of the flow.
\end{abstract}

\section{Introduction}
\label{Sec:Introduction}

The study of coherent structures has been a fundamental area of research in dynamic systems and computational fluid dynamics. Among the various types of coherent structures, Lagrangian coherent structures (LCS) are particularly noteworthy. LCS focuses on identifying regions within a flow that display strong attraction and repulsion of particles over a specific time span. It serves as a valuable tool for analyzing, visualizing, and extracting information from complex dynamical systems. The primary objective is to delineate surfaces of trajectories and separate regions in the phase space that exert significant influence on neighboring trajectories during a defined time interval \cite{hal15}. LCS has found applications in diverse fields, including oceanography \cite{lekleo04,shalekmar05}, meteorology \cite{saphal09}, flight mechanics \cite{carmoh08,tanchahal10}, gravity waves \cite{tanpea10}, and bio-inspired fluid flows \cite{lipmoh09,grerowsmi10,lukyanfau10}. Although the underlying physical intuition behind LCS is straightforward, an accurate and formal mathematical definition capturing this behavior is still a subject of ongoing development. Numerous studies have attempted to establish robust frameworks for identifying LCS structures \cite{hal11,shalekmar05}, but these frameworks often face computational challenges or suffer from certain limitations that affect their precision.

The finite-time Lyapunov exponent (FTLE) is a widely used measure for locating Lagrangian coherent structures (LCS) \cite{halyua00,hal01,hal01b,shalekmar05,lekshamar07}. It quantifies the rate of change in distance between neighboring particles over a finite time interval, considering an infinitesimal perturbation in the initial position. Computing the FTLE field involves several steps. Firstly, the flow map is computed, which establishes the connection between the initial and arrival positions of particles along the characteristic line. The FTLE is then defined based on the largest eigenvalue of the deformation matrix derived from the Jacobian of the flow map. Various Eulerian approaches have been developed to numerically compute the FTLE on a fixed Cartesian mesh in recent studies . For more in-depth discussions on this topic, interested readers can refer to \cite{leu11,leu13,youwonleu17,youleu18,youleu18b,lywn19,ngleu19,youleu20,youleu21,youleu21b} and related references.

Another approach to analyzing trajectories involves trajectory analysis, which aims to develop techniques for understanding and classifying trajectory characteristics. While trajectory clustering is not a new topic, previous methods primarily focused on visualizing or constructing dictionaries and summaries to reveal hidden patterns or predict future routes \cite{leehanwha07,vlagunkol02,gafsmy99,grscg07,wycm11,fkss13,yzzhb17}. These approaches were not specifically designed to discover coherent structures within underlying flows and dynamical systems. However, recent works have explored clustering-based approaches for identifying LCS. Some of these methods involve constructing networks of trajectories and determining edge weights based on trajectory similarity \cite{SchlueterKuck2016CoherentSC,Mowlavi_2022,Rosi2015LagrangianCS}. In \cite{youleu14}, a coherent ergodic partition method was developed to separate trajectories into clusters by integrating and computing long-time averages of functions along the trajectories. This effectively projects high-dimensional data onto a low-dimensional manifold, allowing for partitioning based on these function averages. In \cite{chaleu23}, a clustering-based approach was proposed to identify coherent flow structures in continuous dynamical systems. This method treats particle trajectories over a finite time interval as high-dimensional data points and clusters them from different initial locations into groups using normalized standard deviation or mean absolute deviation to quantify deformation.

The aim of this paper is to propose a general framework based on the trajectory analysis technique to the study of LCS. Specifically, we seek to develop a new metric for efficiently extracting and visualizing LCS using trajectory analysis. Instead of clustering all provided trajectories into groups, we investigate the variation of the trajectories in a small neighborhood of each trajectory. We define a local quantity that measures the similarity of these trajectories from a small local neighborhood based on various possible trajectory metrics. This metric is a commonly used tool for quantifying the similarity between trajectories, and its metric properties may prove useful in developing the theoretical framework \cite{10.1007/s00778-019-00574-9,10.1145/2782759.2782767,DBLP:journals/corr/abs-2004-00722}. Our proposed metric considers the effect of a small perturbation on the entire trajectory over a fixed time period, rather than just the final position, which we believe provides a more robust method for extracting LCS over an extended period. 

This paper is organized as follows: Section \ref{Sec:Preliminaries} presents the related works and background, while Section \ref{Sec:LTV} details the model and algorithm for our proposed quantity. Section \ref{SubSec:Metric} discusses several possible choices of the trajectory metric. Finally, Section \ref{Sec:NumericalExamples} provides several numerical results.


\section{Background}
\label{Sec:Preliminaries}

In this section, we provide several definitions and assumptions on our problem that is crucial for the discussion in the later section. 

\subsection{Trajectory Metric}
Consider a closed and bounded domain $\Omega\subset\mathbb{R}^{n}$, a finite time domain $[t_{0},t_{0}+T]$ and a Lipchitz velocity field $\vv:\Omega\times\mathbb{R}\mapsto\mathbb{R}^{n}$. The movement of any particle in $\Omega$ can be described by the following ordinary differential equation (ODE),
\begin{eqnarray}
\dot{\vx}(t) &=& \vv(\vx(t),t)\quad \forall\vx\in\Omega, t\in [t_{0},t_{0}+T]  \label{Eqn:ODE}
\end{eqnarray}
for any initial condition $\vx(t_0)=\vx_{0}\in\Omega$.

\begin{definition}[Trajectory of Particle]\thlabel{def_traj}
Given $\vx_{0}\in\mathbb{R}^{n}$, a trajectory $\vx:\mathbb{R}\mapsto\Omega$ is defined as the solution of the ODE in (\ref{Eqn:ODE}) with initial condition of $\vx(t_{0})=\vx_{0}$.
\end{definition}

Instead of simply measuring the distance between two trajectories, we are more interested in studying the intrinsic difference in the motion characteristics. One idea is to consider only the so-called displacement trajectory, i.e. the relative displacement from the initial location at the initial time $t=t_0$. 

\begin{definition}[Displacement Trajectory of Particle]
Given $\vx_{0}\in\Omega$ with the trajectory $\vx:\mathbb{R}\mapsto\Omega$ as defined in \thref{def_traj}, a displacement trajectory $\Psi_{\vx_{0}}:\mathbb{R}\mapsto \mathbb{R}^{n}$ is defined as $\Psi_{\vx_{0}}(t):=\vx(t)-\vx(t_{0})$.
\end{definition}

Since we work with a discretized time domain, we consider the following discrete displacement trajectory.

\begin{definition}[Discrete Trajectory]\thlabel{def_discrete}
Given a trajectory $\Psi$, then a $k$-discrete trajectory is defined as $\Psi_{\Delta}:[0,k-1]\cap\mathbb{N}\mapsto \Omega$ such that $\Psi_{\Delta}(j)=\Psi\left(t_{0}+{jT}/{(k-1)}\right)$ for $j=0,1,\cdots,k-1$. 
\end{definition}

We can map any $k$-discrete trajectory into a unique vector in the space of $\mathbb{R}^{kn}$ by 
$$
L(\Psi_{\Delta})=(\pi_{e_{1}}(\Psi_{\Delta}(0)),\pi_{e_{2}}(\Psi_{\Delta}(0)),...,\pi_{e_{d}}(\Psi_{\Delta}(k-1)))
$$ 
where $\pi_{e_{j}}(\vv)$ is the projection map that projects $\vv$ to $e_{j}$ with $e_{j}$ denotes the $j$-th standard basis in $\mathbb{R}^{n}$. One can quickly verify that the map $L$ is linear and bijective. Thus the space of all $k$-discrete trajectory is isomorphic to the space $\mathbb{R}^{kn}$. For the simplicity of presentation, we will abuse the notation of $\Psi_{\Delta}$ for the rest of the article to represent both the mapping $[0,k-1]\cap\mathbb{N}\mapsto \Omega$ and also the vector $L(\Psi_{\Delta})$ whenever it is clear from the context.

To study the similarity between two displacement trajectories, we follow a common approach to define a trajectory metric. Such a method is standard in trajectory analysis, especially in classifying and clustering a set of trajectories \cite{10.1007/s00778-019-00574-9}. The definition is straightforward. Denoting $\Phi$ as the space of all trajectories, we obtain a trajectory metric $d:\Phi\times\Phi\mapsto\mathbb{R}$ such that $(\Phi,d)$ forms a metric space. 

\begin{definition}[Trajectory Metric]
    $\Phi$ is the space of $\mathbf{C}^{1}$ function $\mathbb{R}\mapsto\Omega$, then $d:\Phi\times\Phi\mapsto\mathbb{R}$ is a trajectory metric $\iff$ it satisfies
    \begin{itemize}
        \item $d(\Psi_{1},\Psi_{2})=0\iff  \Psi_{1}=\Psi_{2}\quad\forall\Psi_{1},\Psi_{2}\in\Phi$
        \item $d(\Psi_{1},\Psi_{2})=d(\Psi_{2},\Psi_{1})\quad\forall\Psi_{1},\Psi_{2}\in\Phi$
        \item $d(\Psi_{1},\Psi_{3})\leq d(\Psi_{1},\Psi_{2})+d(\Psi_{2},\Psi_{3})\quad\forall\Psi_{1},\Psi_{2},\Psi_{3}\in\Phi$
    \end{itemize}
\end{definition}

\subsection{Finite-time Lyapunov Exponent (FTLE) and Lagrangian Coherent Structure (LCS)}
\label{SubSec:FTLE}

We define the \textit{flow map} $f:\Omega \rightarrow \Omega$, which collects the solutions to ODEs (\ref{Eqn:ODE}) for all initial conditions $\vx(t_0)=\vx_0\in\Omega$ at all times $t\in\mathbb{R}$. The flow map represents the arrival location $\vx(t_0+T)$ at $t=t_0+T$ of the particle trajectory with the initial condition $\vx(a)=\vx_0$ at the initial time $t=t_0$. In other words, this mapping takes a point from $\vx(t_0)$ at $t=t_0$ to another point $\vx(t_0+T)$ at $t=t_0+T$.

An important quantity in studying the separation rate between adjacent particles with infinitesimal perturbations in their initial locations over a finite-time interval is the finite-time Lyapunov exponent (FTLE). Let us consider the initial time $t=t_0=0$ and the final time $t=T$. The change in the initial infinitesimal perturbation is given by
$$
		\delta \vx(T) =f(\vx+\delta\vx(0))-f(\vx) 
		= \nabla f(\vx)\delta\vx(0)+O(\|\delta\vx(0)\|^2)\,.
$$
The magnitude of the leading-order term of this perturbation is given by 
$$\|\delta\vx(T)\|=\sqrt{\langle \delta\vx(0),[\nabla f(\vx)]^{\prime}\nabla f(\vx)\delta\vx(0) \rangle}$$ where $(\cdot)^{\prime}$ denotes the transpose of a matrix. We denote $\varDelta^T_0(\vx)=[\nabla f(\vx)]^\prime\nabla f(\vx)$ as the Cauchy-Green deformation tensor.

To find the maximum value of $|\delta\vx(T)|$, we align $\delta\vx(0)$ with the eigenvector $\textbf{e}(0)$ associated with the maximum eigenvalue of $\varDelta^T_0(\vx)$. Thus, $\max_{\delta\vx(0)}|\delta\vx(T)|=\sqrt{\lambda_{\max}(\varDelta^T_0(\vx))}|\textbf{e}(0)|=e^{\sigma^T_0(\vx)|T|}|\textbf{e}(0)|$. Therefore, the FTLE $\sigma^T_0(\vx)$ is defined as
\begin{equation}
\sigma^T_0(\vx)=\frac{1}{|T|}\ln\sqrt{\lambda_{\max}(\varDelta^T_0(\vx))} \, .
\label{eq:FTLE}
\end{equation}

Numerical methods for computing the flow map $f(\mathbf{x})$ typically involve ray tracing, where we solve the ODE (\ref{Eqn:ODE}) or the corresponding level set equation \cite{leu11} on a Cartesian mesh at the initial time $t=t_0=0$. Then, for each grid point in the computational domain, we construct the Cauchy-Green deformation tensor by finite differencing the \textit{forward} flow map $f(\mathbf{x}^i_g)$ on the Cartesian mesh at $t=0$. Using equation (\ref{eq:FTLE}), we can compute the \textit{forward} FTLE. Whether we use the Lagrangian (ODE) or the Eulerian (PDE) approach, the FTLE only requires the flow map from the initial location at the initial time to the final arrival location at the final time. It disregards the intermediate trajectory and focuses solely on the particle's initial and final positions. As a result, if an initial patch of particles remains close after a finite period, the FTLE disregards any significant dispersion during the intermediate time and returns a small value. This limitation is unsatisfactory because even if the trajectories of these particles differ significantly, the FTLE fails to capture such unique flow structures. 

It is worth mentioning that while the formulation of FTLE drops the information of the intermediate time steps, it is indeed possible to address this issue by extending the flowmap with an extra time dimension. This approach, as proposed in \cite{youwonleu17}, allows for the preservation of the intermediate time information. However, it is important to note that solving the resulting partial differential equation (PDE) in the extended space can be computationally expensive in terms of both time complexity and space complexity. The high time complexity arises from solving the PDE, while the space complexity increases due to the need to store and manipulate the series of flowmaps. In contrast, our method provides an alternative approach that avoids directly solving the flowmap while still offering a formulation to summarize the information obtained from the intermediate time steps. This alternative approach can be more computationally efficient compared to directly solving the extended flowmap PDE.



\section{Local Trajectory Variation Exponent (LTVE)}
\label{Sec:LTV}

In this section, we formulate our proposed quantity. Since we consider the variations among various local trajectories, we name the quantity the \textit{Local Trajectory Variation (LTV)} and \textit{LTVE} for the corresponding exponent.

\subsection{Theoretical Framework}
We define the Local Trajectory Variation (LTV) and the LTV exponent (LTVE) using a similar approach as the FTLE. Consider the metric space $(\Phi, d)$, where $\Phi$ is the space of all trajectories and $d$ is the trajectory metric. Let $\Gamma: \Omega \mapsto \Phi$ be the mapping that assigns each particle to its trajectory. Therefore, $\Gamma(\mathbf{x}_0) = \Psi_{\mathbf{x}_0}$. To capture only the intrinsic shape difference in trajectories, disregarding the effect of physical distance due to variations in initial positions, we require that the mapping $\Gamma$ instead maps to the displacement trajectory. Following a similar approach as the FTLE, we perturb the input to $\Gamma$ and measure the perturbation effect using the trajectory metric $d\left(\Gamma(\mathbf{x}_0 + \delta), \Gamma(\mathbf{x}_0)\right)$. Formally:

\begin{definition}[Local Trajectory Variation]
Given a trajectory field $\Gamma: \Omega \mapsto \Phi$ and a small $\delta \in \mathbb{R}^+$, denote $\mathcal{B}^{\Omega}_{\delta}(\vx)=\{\vu\in\mathbb{R}^{n}|\norm{\vu}=1\land\vx+\delta\vu\in\Omega\}$, we define the LTV at any $\mathbf{x} \in \Omega$ as
$$
\text{LTV}(\vx;\delta,T)=\max_{\vu\in\mathcal{B}^{\Omega}_{\delta}(\vx)}\;d(\Gamma(\vx+\delta\vu),\Gamma(\vx)) \, .
$$
\end{definition}

\begin{definition}[Local Trajectory Variation Exponent]
    Given $\delta \in \mathbb{R}^+$ and the LTV function LTV$(\vx;\delta,T)$, the LTVE$:\Omega\mapsto\mathbb{R}$ is defined as
    $$\text{LTVE}(\vx;\delta,T)=\frac{1}{\abs{T}}\ln\left(\frac{1}{\delta}\text{LTV}(\vx;\delta,T)\right)$$
\end{definition}

\subsection{Relation with FTLE}
\label{SubSec:FTLEvsLTVE}
Intuitively, the flow map introduced in the FTLE can be viewed as a $2$-trajectory, a trajectory of length $2$, and thus the FTLE can be viewed as a special case of LTVE. Below, we provide a rigorous linkage between FTLE and LTVE. Let $d$ be the $l^{2}$-norm distance, $\Gamma$ be the discrete $2$-trajectory mapping, $\sigma(\vx)$ be the FTLE field, and $f$ be the flow map induced in the construction of FTLE. 
Let $\Lambda_{f}:\Omega\mapsto\R$ be the following function:
$$
\Lambda_{f}(\vx):=\max_{\lambda} \{\lambda\in\R|\exists\vu\in\R^{n}\setminus\{\mathbf{0}\},\Delta^{T}_{0}(\vx)\vu=\lambda\vu\} \, .
$$
Then we start by showing the following lemma.
\begin{lemma}\label{lem:1}
    If $f\in C^{1}$ on $\Omega$, then $\Lambda_{f}$ is continuous on $\Omega$.
\end{lemma}

\begin{proof}
To begin with, we note that the function $\Lambda_{f}$ is well-defined since the deformation tensor $\Delta^{T}_{0}$ is symmetric. Now, we have
\begin{align*}
\Lambda_{f}(\vx+\vu)-\Lambda_{f}(\vx)&=\max_{\vv\in\mathcal{B}^{\Omega}_{\delta}(\vx)}\norm{\nabla f(\vx+\vu)\cdot \vv}^{2}-\max_{\vv\in\mathcal{B}^{\Omega}_{\delta}(\vx)}\norm{\nabla f(\vx)\cdot \vv}^{2}\\
    &\leq \max_{\vv\in\mathcal{B}^{\Omega}_{\delta}(\vx)}\left\{\norm{\nabla f(\vx+\vu)\cdot \vv}^{2}-\norm{\nabla f(\vx)\cdot \vv}^{2}\right\}\\
    &=\max_{\vv\in\mathcal{B}^{\Omega}_{\delta}(\vx)}\left\{(\norm{\nabla f(\vx+\vu)\cdot \vv}+\norm{\nabla f(\vx)\cdot \vv})(\norm{\nabla f(\vx+\vu)\cdot \vv}-\norm{\nabla f(\vx)\cdot \vv})\right\} \, .
\end{align*}
Since $\Omega$ is closed and bounded, and $f\in C^{1}$ implies $\norm{\nabla f(\vx)}<L$ for some $L\in\mathbb{R}$, we have
$$
\Lambda_{f}(\vx+\vu)-\Lambda_{f}(\vx)\leq 2L\max_{\vv\in\mathcal{B}^{\Omega}_{\delta}(\vx)}\norm{\nabla f(\vx+\vu)\cdot \vv-\nabla f(\vx)\cdot \vv}\leq 2L\norm{\nabla f(\vx+\vu)-\nabla f(\vx)} \, .
$$
Similarly, we can obtain another inequality $\Lambda_{f}(\vx)-\Lambda_{f}(\vx+\vu)\leq 2L\norm{\nabla f(\vx+\vu)-\nabla f(\vx)}$. These two inequalities imply that $\abs{\Lambda_{f}(\vx+\vu)-\Lambda_{f}(\vx)}\leq 2L\norm{\nabla f(\vx+\vu)-\nabla f(\vx)}$. Then it follows that $f\in C^{1}$ implies the latter term can be bounded arbitrarily small.
\end{proof}

From the definition of FTLE, we can express the FTLE field using the greatest eigenvalue $\lambda_{\max}(\vx)$ of $\Delta_{0}^{T}(\vx)$. In particular, we have $\sigma(\vx)=\frac{1}{2\abs{T}}\ln\lambda_{\max}$. Hence, for the FTLE value to be well-defined over $\Omega$, we require $\lambda_{\max}>0$ at all $\vx\in\Omega$. Then, it follows from Lemma \ref{lem:1} that $\lambda_{*}:=\min_{\vx\in\Omega}\Lambda_{f}(\vx)$ is well-defined and satisfies $0\leq\lambda_{*}\leq\lambda_{\max}(\vx)$ for all $\vx\in\Omega$. Since we assumed the mapping maps to the displacement trajectory $\Psi$, any $2$-trajectory evaluated under such a norm will give the particle displacement at time $t=t_{0}+T$. Namely,
$$
\text{LTV}(\vx)=\max_{\vu\in\mathcal{B}^{\Omega}_{\delta}(\vx)}\;d(\Gamma(\vx+\delta\vu),\Gamma(\vx))=\max_{\vu\in\mathcal{B}^{\Omega}_{\delta}(\vx)}\norm{f(\vx+\delta\vu)-f(\vx)-\delta\vu}_{2} \, .
$$
For simplicity, we denote $\mathcal{F}_{\delta}(\vx):=\max_{\vu\in\mathcal{B}^{\Omega}_{\delta}(\vx)}\left\{\norm{f(\vx+\delta\vu)-f(\vx)}_{2}\right\}$. Recall from the construction of FTLE, we have $\mathcal{F}_{\delta}(\vx)\approx \sqrt{\lambda_{\max}}\delta\geq \sqrt{\lambda_{*}}\delta$. Thus, we can approximate the FTLE as $\sigma(\vx)\approx\frac{1}{\abs{T}}\ln\frac{\mathcal{F}_{\delta}(\vx)}{\delta}$. Then consider
\begin{align*}
    \text{LTV}(\vx)^{2}&=\left[\max_{\vu\in\mathcal{B}^{\Omega}_{\delta}(\vx)}\norm{f(\vx+\delta\vu)-f(\vx)-\delta\vu}_{2}\right]^{2}
    =\max_{\vu\in\mathcal{B}^{\Omega}_{\delta}(\vx)}\norm{f(\vx+\delta\vu)-f(\vx)-\delta\vu}_{2}^{2}\\
    &=\max_{\vu\in\mathcal{B}^{\Omega}_{\delta}(\vx)}\left\{\norm{f(\vx+\delta\vu)-f(\vx)}_{2}^{2}+\delta^{2}\norm{\vu}^{2}_{2}-2\delta\inprod{f(\vx+\delta\vu)-f(\vx),\vu}\right\}\\
    &=\max_{\vu\in\mathcal{B}^{\Omega}_{\delta}(\vx)}\left\{\norm{f(\vx+\delta\vu)-f(\vx)}_{2}\left(\norm{f(\vx+\delta\vu)-f(\vx)}_{2}-2\delta\norm{\vu}_{2}\cos\theta\right)\right\}+\delta^{2}
\end{align*}
where $\theta$ is determined by the angle between $f(\vx+\delta\vu)-f(\vx)$ and $\vu$. Taking the trivial bound of $\cos\theta\geq -1$, we obtain
$\text{LTV}(\vx)^{2}\leq \mathcal{F}_{\delta}(\vx)\left(\mathcal{F}_{\delta}(\vx)+2\delta\right)+\delta^{2}$, and therefore,
$$
\frac{\text{LTV}(\vx)^{2}}{\mathcal{F}_{\delta}(\vx)^{2}}\leq 1+\frac{\delta^{2}}{\mathcal{F}_{\delta}(\vx)^{2}}+2\frac{\delta}{\mathcal{F}_{\delta}(\vx)}\implies \frac{\text{LTV}(\vx)}{\mathcal{F}_{\delta}(\vx)}\leq 1+\frac{\delta}{\mathcal{F}_{\delta}(\vx)} \, .
$$
Since the logarithmic function is increasing, we have
$$
\ln\frac{\text{LTV}(\vx)}{\mathcal{F}_{\delta}(\vx)}\leq \ln\left(1+\frac{\delta}{\mathcal{F}_{\delta}(\vx)}\right)\leq\ln\left(1+\frac{1}{\sqrt{\lambda_{*}}}\right) \, .
$$
On the other hand, by considering $g(\vx):=f(\vx)-\vx$, we can rewrite LTV in a different form, $\text{LTV}(\vx)=\max_{\vu\in\mathcal{B}^{\Omega}_{\delta}(\vx)}\norm{g(\vx+\delta\vu)-g(\vx)}\approx\max_{\vu\in\mathcal{B}^{\Omega}_{\delta}(\vx)}\norm{\nabla g\cdot\delta\vu}$. Then since $f\in C^{1}$, we have $g\in C^{1}$. Hence following the same approach, we can define $\eta_{*}:=\min_{\vx\in\Omega}\Lambda_{g}(\vx)$ and thus $\text{LTV}(\vx)\geq \sqrt{\eta_{*}}\delta$ for all $\vx\in\Omega$. Hence we have 
$\ln\left[ {\mathcal{F}_{\delta}(\vx)}/ \text{LTV}(\vx) \right] \leq\ln \left(1+{1}/{\sqrt{\eta_{*}}} \right)$. Thus it follows from
$$
\abs{\text{LTVE}(\vx)-\sigma(\vx)}=\frac{1}{\abs{T}}\abs{\ln\left(\frac{\text{LTV}(\vx)}{\delta}\right)-\ln\left(\frac{\mathcal{F}_{\delta}(\vx)}{\delta}\right)}=\frac{1}{\abs{T}}\abs{\ln\frac{\text{LTV}(\vx)}{\mathcal{F}_{\delta}(\vx)}} \, ,
$$
we arrive our final conclusion that
\begin{equation}
    \abs{\text{LTVE}(\vx)-\sigma(\vx)}\leq\frac{1}{\abs{T}}\max\left\{\ln\left(1+\frac{1}{\sqrt{\eta_{*}}}\right),\ln\left(1+\frac{1}{\sqrt{\lambda_{*}}}\right)\right\} 
    \label{Eqn:LTVEvsFTLE}
\end{equation}
for all $\vx\in\Omega$.

The introduction of $\lambda_*$ and $\eta_*$ in equation (\ref{Eqn:LTVEvsFTLE}) aims to establish a universal bound with respect to the domain $\Omega$. By definition, $\lambda_*$ and $\eta_*$ represent the lower bounds of $\Lambda_f$ and $\Lambda_g$, respectively. This implies that for each point in the domain $\Omega$, there exists a direction of perturbation that induces a change in the trajectory. Considering this condition is natural, as the FTLE should be well-defined at a point when the maximal eigenvalue of the deformation tensor is positive, indicating trajectory deviation when perturbed. Furthermore, allowing the bounds to depend on $\mathbf{x}$, the location in the domain, can yield tighter bounds. Intuitively, if we evaluate the bound near the LCS, the corresponding maximal eigenvalue should be large, resulting in a smaller bound between the LTVE and FTLE quantities. Ideally, our quantity can capture all structures that the FTLE captures by applying the $l^2$-norm with two trajectories.

Instead of solely considering the takeoff and arrival locations of a trajectory, we have expanded our investigation by adding a third point to each trajectory data. This addition allows us to explore longer trajectories and their impact on our quantity. In Section \ref{SubSec:3trajEval}, we will present an example that demonstrates the superiority of our quantity over the FTLE when considering these longer trajectories. This example will showcase a significant difference between the LTVE with three trajectories and the FTLE, emphasizing the advantages of our approach.

\subsection{Numerical Scheme and Error Analysis}
To perform the numerical computation of the quantity, it is necessary to discretize and apply relaxation techniques to make the problem computationally feasible. Below, we provide a brief summary of these processes and analyze their associated errors.

First, we focus on the discretization in the time domain. We employ a simple scheme to approximate the real trajectory. We uniformly sample $k$ points along the trajectory, where $k$ can be adjusted either adaptively or manually depending on the problem. The trajectory is then approximated by a polygonal line connecting these sampling points, which essentially provides a linear approximation. While this approximation may introduce imprecise results when $k$ is poorly chosen, opting for a higher-order approximation would increase the complexity of designing and implementing the trajectory comparison metric.

Now, we analyze the error introduced by this approximation. For any continuous $\Psi$, we denote $\Psi_{\Delta}$ as the discrete trajectory approximated by the polygonal line. Similarly, for the continuous trajectory metric $d$, we denote $d_{\Delta}$ as its discrete counterpart. For any (continuous) trajectories $\Psi_{1}$ and $\Psi_{2}$, using metric properties, we have:
$$
d(\Psi_{1},\Psi_{2})\leq d(\Psi_{1},\Psi_{1,\Delta})+d(\Psi_{1,\Delta},\Psi_{2,\Delta})+d(\Psi_{2},\Psi_{2,\Delta})
$$
and
$$
d(\Psi_{1,\Delta},\Psi_{2,\Delta})\leq d(\Psi_{1},\Psi_{1,\Delta})+d(\Psi_{1},\Psi_{2})+d(\Psi_{2},\Psi_{2,\Delta}) \, .
$$ 
Combining them, we have $\abs{d(\Psi_{1},\Psi_{2})-d(\Psi_{1,\Delta},\Psi_{2,\Delta})}\leq d(\Psi_{1},\Psi_{1,\Delta})+d(\Psi_{2},\Psi_{2,\Delta})$. Hence the error can be separated into two parts,
\begin{align*}
    \abs{d(\Psi_{1},\Psi_{2})-d_{\Delta}(\Psi_{1,\Delta},\Psi_{2,\Delta})}&\leq \abs{d(\Psi_{1},\Psi_{2})-d(\Psi_{1,\Delta},\Psi_{2,\Delta})}+\abs{d(\Psi_{1,\Delta},\Psi_{2,\Delta})-d_{\Delta}(\Psi_{1,\Delta},\Psi_{2,\Delta})}\\
    &\leq d(\Psi_{1},\Psi_{1,\Delta})+d(\Psi_{2},\Psi_{2,\Delta})+\abs{d(\Psi_{1,\Delta},\Psi_{2,\Delta})-d_{\Delta}(\Psi_{1,\Delta},\Psi_{2,\Delta})} \, .
\end{align*}
Taking the supremum over the space of trajectories, we have the error bound
$$
2\sup_{\Psi}\{d(\Psi_{1},\Psi_{1,\Delta})\}+\sup_{\Psi_{1,\Delta},\Psi_{2,\Delta}}\abs{d(\Psi_{1,\Delta},\Psi_{2,\Delta})-d_{\Delta}(\Psi_{1,\Delta},\Psi_{2,\Delta})} \, .
$$
The first part addresses the error resulting from estimating the trajectory using a polygonal line under the desired metric. As mentioned earlier, this error arises from the use of linear approximation for trajectories. One can reduce this term by increasing the number of sampling points, adjusting the parameter $k$, or employing a higher-order approximation scheme to achieve a more accurate solution. The second part of the error analysis stems from the discretization of the metric itself. Approximating the metric using its discrete variant introduces additional error. For metrics such as the $L^{p}$-norm, one can reduce this term to arbitrary precision by employing a higher order of numerical integration. Explicitly considering the discretization of the metric is important, especially for metrics like the Fr\'{e}chet distance and Hausdorff metric. The computational complexity of these metrics in the continuous setting is prohibitively high for practical implementation. Therefore, including an approximation of the metric itself becomes necessary. The actual order of the error depends on the chosen metric and underlying field, as mentioned previously. Providing a general bound on the error without specifying the metric is challenging. In our chosen metric for the experiment, all these errors are well-bounded. The derivation of these bounds is either standard or well-studied in the literature, and we omit it here for brevity.

Next, we focus on the relaxation in the space domain. It is important to note that solving for the optimal $\mathbf{u}$ can be computationally expensive. However, we can adopt an approach similar to the FTLE by constructing the Cauchy-Green deformation tensor for $T$. It should be noted that this process is more complex compared to the FTLE case, given the higher dimensionality involved. Nevertheless, as we will see in later numerical experiments, our relaxation scheme is sufficient for extracting the LCS.

Our relaxation scheme is straightforward. Instead of comparing the trajectories within the $\delta$-unit circle, we only compare them to nearby trajectories. The specific choice of neighborhood depends on the scheme specified by the users. Below, we propose several possible schemes for selecting the neighborhood:
\begin{definition}[First Order Relaxed-LTV]
    Given a trajectory field $\Gamma:\Omega\mapsto\Phi$, fix some small $\delta\in\mathbb{R}^{+}$, denote $e_{j}$ to be the $j$-th standard basis in $\mathbb{R}^{n}$, then $\forall \vx\in\Omega$, we define
    $$\text{RLTV}_{1}(\vx)=\max_{e_{j}\in\mathcal{B}^{\Omega}_{\delta}(\vx)}\max\left\{d(\Gamma(\vx+\delta e_{j}),\Gamma(\vx)),d(\Gamma(\vx-\delta e_{j}),\Gamma(\vx))\right\}\, .$$
\end{definition}

\begin{definition}[Second Order Relaxed-LTV]
    Given a trajectory field $\Gamma:\Omega\mapsto\Phi$, fix some small $\delta\in\mathbb{R}^{+}$, denote $e_{j}$ to be the $j$-th standard basis in $\mathbb{R}^{n}$, then denote
    $\mathcal{N}:=\{\vu\in\mathbb{R}^{n}|\vu=\pm a_{1}e_{1}\pm a_{2}e_{2}\pm...\pm a_{n}e_{n},\,a_{i}=\{0,1\}\enspace\forall i\}$.
    Then $\forall \vx\in\Omega$, we define
    $$\text{RLTV}_{2}(\vx)=\max_{\vu\in \mathcal{N}\cap\mathcal{B}^{\Omega}_{\delta}(\vx)}d(\Gamma(\vx+\delta \vu),\Gamma(\vx)) \, .$$
\end{definition}

\begin{definition}[Third Order Relaxed-LTV]
    Given a trajectory field $\Gamma:\Omega\mapsto\Phi$, fix some small $\delta\in\mathbb{R}^{+}$, denote $e_{j}$ to be the $j$-th standard basis in $\mathbb{R}^{n}$, then denote
    $\mathcal{N}:=\{\vu\in\mathbb{R}^{n}|\vu=\pm a_{1}e_{1}\pm a_{2}e_{2}\pm...\pm a_{n}e_{n},\,a_{i}=\{0,1\}\enspace\forall i\}$
    and
    $\mathcal{N'}:=\{\vv\in\mathbb{R}^{n}|\vv=\vu+e_{i},\,\forall \vu\in\mathcal{N},\,\forall i\}\setminus\{2e_{i}\in\mathbb{R}^{n}|\forall i\} \, .$
    Then $\forall \vx\in\Omega$, we define
    $$\text{RLTV}_{3}(\vx)=\max_{\vu\in \mathcal{N'}\cap\mathcal{B}^{\Omega}_{\delta}(\vx)}d(\Gamma(\vx+\delta \vu),\Gamma(\vx)) \, .$$
\end{definition}

The corresponding RLTVE are defined similarly as in the LTVE case and we omit it here for brevity. While it is indeed possible to adopt a random sampling approach, such as Monte-Carlo sampling, to approximate the optimal $\mathbf{u}$ without using relaxation, we have chosen to utilize the relaxation scheme proposed for improving time complexity. The main advantage of the relaxation scheme is the ability to reuse trajectories. Since all the comparison points lie on the mesh points, it is sufficient to compute the trajectories only at the mesh points. This significantly reduces the computational complexity. If we were to use random sampling, the complexity would scale linearly with the sampling size. By utilizing the relaxation scheme and working with the mesh points, we can achieve better computational efficiency compared to random sampling approaches.


\subsection{Algorithm}
Below we present the main algorithm for computing the RLTVE field and analyze its time and space complexity. Given a rectangular computation domain $\Omega=[x_{1,0},x_{1,1}]\times...\times [x_{n,0},x_{n,1}]$ and the Lipchitz velocity field $\vv:\Omega\mapsto\mathbb{R}^{n}$, fix some small $\delta\in\mathbb{R}^{+}$ for the relaxation condition, then we adapt a uniform mesh such that each particles are separated with distance $\delta$ along each standard axis direction.

\begin{algorithm*}[h!]
    \DontPrintSemicolon
    \caption{RLTVE}
    \SetKwFunction{RLTVE}{RLTVE}
    \SetKwProg{Fn}{Function}{:}{end}
    \Fn{\RLTVE{$\vv:\Omega\mapsto\mathbb{R}^{n}$,$\delta\in\mathbb{R}^{+}$}}{
        $M\leftarrow$ uniform meshes with $\Delta x_{j}=\delta$ for all $j\in[1,d]$\\
        $\Psi\leftarrow$ Solve the trajectory field from $\vv$ using ODE solver with $M$ as meshes\\
        $\Psi\leftarrow$ transform to displacement trajectory\\
        $RLTVE\leftarrow$ Int array of equal dimension as $M$ with initial value $0$\\
        \For{$\vx\in M$}{
            $\mathcal{N}\leftarrow$ neighbourhood of $\vx$ defined by the chosen scheme.\\
            $RLTVE[\vx]\leftarrow\frac{1}{\abs{T}}\ln\frac{1}{\delta}\max_{\vu\in\mathcal{N}\cap\mathcal{B}^{\Omega}_{\delta}(\vx)}d(T(\vx+\delta\vu),T(\vx))$
        }
        \Return{$RLTVE$}
    }
\end{algorithm*}

\subsection{Implementation Details and Analysis}

In this section, we summarize some technical implementation details and provide some analysis of the algorithm's performance. We will consider both the computational time complexity and the algorithm's memory requirement. 

\paragraph{Boundary Condition.} Upon solving the trajectories using the given velocity field, it is possible for particles to exit the computational domain $\Omega$. In cases where no prior information is available outside $\Omega$, we need to impose artificial boundary conditions on the numerical integrator. One approach, as suggested in \cite{doi:10.1063/1.3276061}, is to implement a velocity extension algorithm. This method creates an artificial layer that gradually stops the particles, similar to the absorbing boundary condition used in solving the wave equation. However, in our work, instead of exploring the effects of the velocity extension, we directly enforce a stopping boundary condition. When the numerical solution from the ODE leaves the computational domain, we halt the numerical integrator and project the solution onto the boundary of the computational domain $\partial \Omega$. We will provide more details on this effect in Section \ref{SubSec:ArtificialRidge}.

\paragraph{Time Complexity.} Now let us analyze the time complexity. First, we note that using uniform meshes $M$ is not necessary for actual implementation. All particle positions can be computed using looping indices with a given $\delta$. Therefore, the time complexity mainly comes from the ODE solver and the computation of RLTVE. However, it is possible to pre-compute and store trajectory data, reading it from a file during computation. Still, it takes $O\left({\abs{\Omega}}{\delta^{-n}}k\right)$ time to write all the values for each trajectory in the domain. Next, for computing the RLTVE, although depending on the scheme of relaxation chosen, the comparison done for each particle can be approximated as $O(n)$. It is important to note that the complexity of each computation depends on the metric $d$, and it is generally not negligible as the number of sampling points in each trajectory increases. Let us denote the time complexity for computing a single query of the metric as $T_{d}(n,k)$. The total time complexity for computing the RLTVE would then be $O\left({\abs{\Omega}}{\delta^{-n}}nT_{d}(n,k)\right)$. It is worth mentioning that for the trajectory metric applied in this project, it takes at least $\Omega(k)$ time to compute. Furthermore, since comparing two arbitrary trajectories without any special constraints requires going through each sampling point at least once, we can assume that $T_{d}(n,k)\in\Omega(k)$. Hence the complexity is at least $O\left({\abs{\Omega}}{\delta^{-n}}n k\right)$ but could be potentially higher. In summary, the total time complexity for general metric is given by $O\left({\abs{\Omega}}{\delta^{-n}}nT_{d}(n,k)\right)$.

\paragraph{Storage.} Next, we consider the memory required for the algorithm. A naive approach of saving all trajectories and the computed RLTVE values at each point would take space of $O\left({\abs{\Omega}}{\delta^{-n}}(k+1)\right)$. However, we can reduce it to $O\left({\abs{\Omega'}}{\delta^{-(n-1)}}k\right)$ without affecting the time complexity, where $\abs{\Omega'}={\abs{\Omega}}/{\max_{j\in [1,n]}\abs{x_{j,1}-x_{j,0}}}$. The reduction in space usage is based on the observation that our calculations are based on comparisons with the neighborhood. Therefore, we can discard trajectory data that is too far from the current computing position. To achieve this, we first select an axis along which $\Omega$ has its longest side. We then compute and store all the trajectories on the co-dimension one faces that are tangential to the chosen longest axis and touch one of the boundaries of the corresponding side. Next, we move to the next set of faces, shifting $\delta$ from the boundary of the longest side. Using the previously calculated values, we can reduce the number of calculations required in these faces by $1$. After computing all the faces, we delete the stored values from the previous layer and replace them with the values calculated in the current layer. With this approach, we only need to store at most $O\left({\abs{\Omega'}}{\delta^{-(n-1)}}(k+1)\right)$, which is on the surface. Furthermore, for the RLTVE values, they can be computed and written to the output file on-the-fly during the trajectory computation. This further reduces our memory complexity to $O\left({\abs{\Omega'}}{\delta^{-(n-1)}}k\right)$.


\subsection{Trajectory Metric}
\label{SubSec:Metric}
In this section, we specify our choice of metric. In this paper, we investigate three different trajectory metrics, and we are going to give a brief definition of each below. In the following discussion, $\Phi$ denotes the time-continuous trajectory space, and $\Phi_{\Delta}$ represents the space of all $k$-discrete trajectory.

\subsubsection{$L^{2}$-norm}
As might be expected that the most natural choice for trajectory metric will be the $L^{2}$-norm, it is a well-studied norm in classic analysis and serve as a good intuitive for understanding the physical meaning of the value captured by our quantity. Moreover it is easy to be computed, in particular its discrete variant, $l^{2}$ norm, has a linear complexity and is easy to implement.
\begin{definition}[Normalized $L^{2}$-norm]
    The normalized $L^{2}$-norm $\norm{\cdot}_{\overline{L}^{2}}:\Phi\mapsto\mathbb{R}$ is defined as
    $\norm{\Psi}_{\overline{L}^{2}}^{2}=\frac{1}{\abs{T}}\int_{t_{0}}^{t_{0}+T}\abs{\Psi(t)}^{2}\,dt$.
    For the discrete variant, we naturally take the discrete analog of $L^{2}$-norm, we define the normalized $l^{2}$-norm as $\norm{\cdot}_{\overline{l^{2}}}:\Phi_{\Delta}\mapsto\mathbb{R}$,
    $\norm{\Psi_{\Delta}}_{\overline{l^{2}}}^{2}=\frac{1}{k}\norm{\Psi_{\Delta}}_{l^{2}}^{2}$
    where $\norm{\cdot}_{l^{2}}$ denote the conventional $l^{2}$-norm on $\mathbb{R}^{kn}$, such definition is safe by recalling the isomorphism of $\mathbb{R}^{kn}$ and $\Phi_{\Delta}$ as mentioned in \thref{def_discrete}.
\end{definition}

\subsubsection{Fr\'{e}chet Distance}
Fr\'{e}chet distance is a quantity that captures the minimal maximum separation of trajectories over reparameterization in the time domain. It is a popular similarity measure of trajectories used in the field of trajectory analysis and clustering. This distance measure considers the reparameterization of the trajectory in the time-domain, which makes it interesting to study, as such a property might be able to capture structures that the classical method fails to identify.
\begin{definition}[Fr\'{e}chet Distance]
    Denote $\mathcal{P}$ as the space of function $f:[t_{0},t_{0}+T]\mapsto[t_{0},t_{0}+T]$ which satisfies $\forall f\in\mathcal{P}$, $f$ is continuous, surjective and is non-decreasing.
    Then the Fr\'{e}chet distance $d_{F}:\Phi\times\Phi\mapsto\mathbb{R}$ is defined as
    $$d_{f}(\Psi_{1},\Psi_{2})=\inf_{\alpha,\beta\in\mathcal{P}}\max_{\tau\in[t_{0},t_{0}+T]}\norm{\Psi_{1}(\alpha(\tau))-\Psi_{2}(\beta(\tau))}_{2}$$
    where $\norm{\cdot}_{2}$ is the $l^{2}$ norm induced by the space $\mathbb{R}^{n}$.
\end{definition}

Unfortunately, there are certain issues with the Fr\'{e}chet distance. The first issue is that it does not meet the requirement of being a metric. Specifically, it is only a pseudo-metric, which means it fails to meet the condition that $d(a,b)=0\iff  a=b$. However, we can relax this condition and still use the Fr\'{e}chet distance by using its discrete variant. The work \cite{Eiter1994ComputingDF} defined the discrete variant of the Fr\'{e}chet distance and showed that it preserves the metric property. It also serves as a good approximation to the Fr\'{e}chet distance with respect to the sampling size, resolving the first issue. The second issue is the time complexity. The discrete Fr\'{e}chet distance can be computed using dynamic programming, which has a runtime of $O(nm)$, where $n$ and $m$ are the lengths of the two trajectories. In \cite{DBLP:journals/corr/Bringmann14}, it is proven that for both continuous and discrete Fr\'{e}chet distance, there is no strong subquadratic algorithm. This means that there is no algorithm with a complexity of $O(n^{2-\delta})$ for some $\delta>0$, unless the strong exponential time hypothesis (SETH) fails. SETH is a long-standing hypothesis in the study of algorithmic complexity theory. It states that there is no algorithm that can solve the $k$-SAT problem with length $n$ in $O((2-\delta)^{n})$ time $\forall \delta>0$. To address this issue, there are several algorithms that have been proposed. The state-of-the-art solutions include an exact algorithm proposed by \cite{DBLP:journals/corr/abs-1204-5333} that runs in weak subquadratic time, as well as an almost-linear time approximation algorithm proposed by \cite{CHAN201872}. These algorithms provide efficient ways to compute the Fr\'{e}chet distance. In this project, we have implemented the $O(k^{2})$ solution using dynamic programming since the efficiency is not the main focus of studying this metric, we are more interested in whether it produces different results from the $L^2$-norm.

\begin{definition}[Discrete Fr\'{e}chet Distance]\enspace
    For all $\Psi_{\Delta,1},\Psi_{\Delta,2}\in\Phi_{\Delta}$, the coupling of $\Psi_{\Delta,1},\Psi_{\Delta,2}$ is defined as
    $$L(\Psi_{\Delta,1},\Psi_{\Delta,2}):=(\Psi_{\Delta,1}(a_{0}),\Psi_{\Delta,2}(a_{0})),(\Psi_{\Delta,1}(a_{1}),\Psi_{\Delta,2}(b_{1})),...,(\Psi_{\Delta,1}(a_{k-1}),\Psi_{\Delta,2}(a_{k-1}))$$
    and satisfying  $a_{k-1}=b_{k-1}=k-1$ with $a_{0}=b_{0}=0$ and $(a_{j+1}=a_{j}+1\land b_{j}=b_{j+1})\lor (b_{j+1}=b_{j}+1\land a_{j}=a_{j+1})\quad\forall j\in[0,k-2]$.
    Then for each coupling, its length $\norm{L}$ is defined as
    $\norm{L}=\max_{i\in[0,n]}\norm{\Psi_{\Delta,1}(a_{0})-\Psi_{\Delta,2}(a_{0})}_{2}$
    then the discrete Fr\'{e}chet distance $d_{F_{\Delta}}$ is defined as
    $d_{F_{\Delta}}(\Psi_{\Delta,1},\Psi_{\Delta,2})=\min\{\norm{L}\in\mathbb{R}|L\text{ is a coupling of $\Psi_{\Delta,1}$ and $\Psi_{\Delta,2}$}\}$.
\end{definition}

\subsubsection{Hausdorff Metric}
The last metric we have chosen is the Hausdorff metric, which is another important metric for comparing the similarity of two different trajectories or sets of points. Unlike other metrics, the Hausdorff metric does not consider the ordering of points in the set. It is worth noting that we are not the first to apply the Hausdorff metric for computing the LCS. In previous works \cite{TYLER2022101883,BLAZEVSKI201446}, they also incorporated the Hausdorff metric into their frameworks for computing the LCS. The Hausdorff metric can be naively computed in $O(k^{2})$ time complexity. However, similar to the previous metric, our main focus is not on the runtime of this metric but rather on its difference from the $L^2$-norm.

\begin{definition}[Hausdorff Metric]
    The one sided Hausdorff metric $d_{H}':\Phi\times\mathbb{R}^{n}\mapsto\mathbb{R}$ is defined as 
    $d_{H}'(\Psi,\psi)=\inf_{\tau\in[t_{0},t_{0}+T]}\norm{\Psi(\tau)-\psi}_{2}$
    where $\norm{\cdot}_{2}$ is the $l^{2}$ norm induced by the space $\mathbb{R}^{n}$.\\
    Then the Hausdorff metric is defined as
    $$d_{H}(\Psi_{1},\Psi_{2})=\max\left\{\sup_{t\in[t_{0},t_{0}+T]}d_{H}'(\Psi_{1},\Psi_{2}(t)),\sup_{t\in[t_{0},t_{0}+T]}d_{H}'(\Psi_{2},\Psi_{1}(t))\right\} \, .$$
    Noted that by simply changing the domain of $\sup$ and $\inf$ correspondingly to $[0,k-1]$, the above definition works for $k$-discrete trajectory as well.
\end{definition}

\section{Numerical Examples}
\label{Sec:NumericalExamples}

This paper specifically focuses on two-dimensional flows. Extending the algorithm to three-dimensional examples is relatively straightforward, but we will not cover that in this paper. The trajectories in all the test cases have been pre-computed using the fourth-order Runge-Kutta method, with the underlying velocity known analytically.  For the implementation of the main computation algorithm, we have used the programming language \textsf{C++}. Additionally, for visualization purposes, we have utilized the \textsf{contourf} function in \textsf{MATLAB} for simplicity. In all of the experiments, unless otherwise specified, we used 100 sampling points for the trajectories.

\subsection{Circular Flow}
\label{SubSec:Circular}

\begin{figure}[h!]
\centering
(a)\includegraphics[trim=50 20 50 20, clip, width=0.45\linewidth]{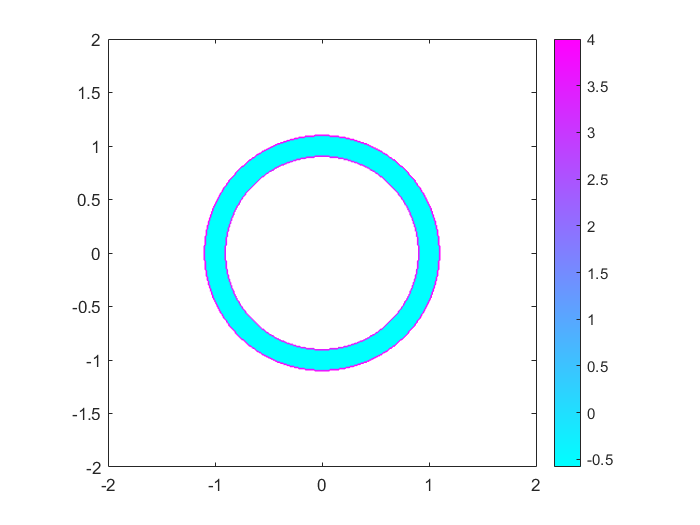}
(b)\includegraphics[trim=50 20 50 20, clip, width=0.45\linewidth]{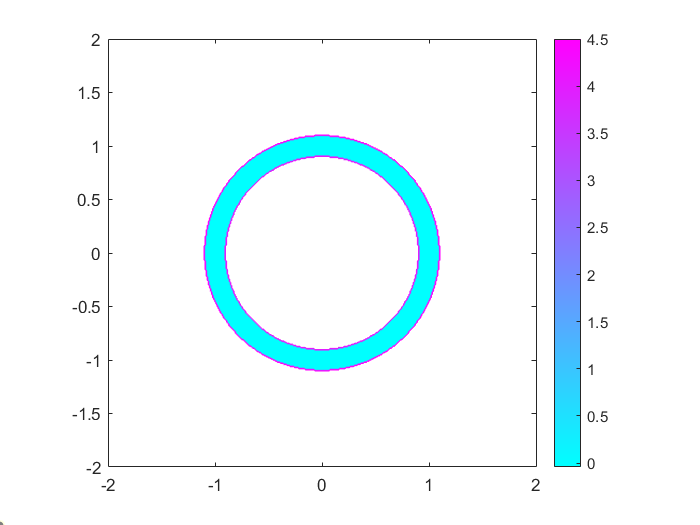}
(c)\includegraphics[trim=50 20 50 20, clip, width=0.45\linewidth]{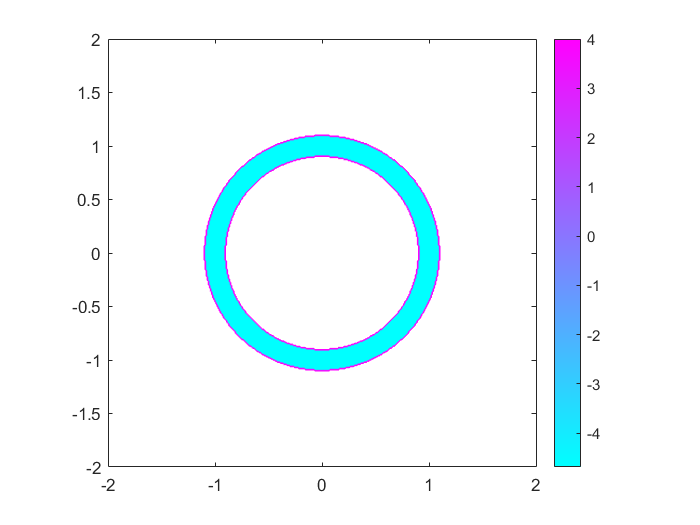}
(d)\includegraphics[trim=50 20 50 20, clip, width=0.45\linewidth]{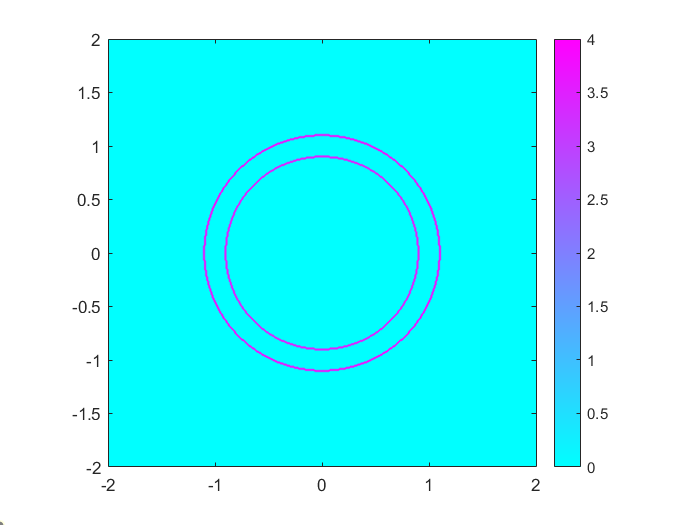}
\caption{(Section \ref{SubSec:Circular}) The computed LTVE using (a) the normalized 2-norm, (b) the Fr\'{e}chet distance, and (c) the Hausdorff metric. (d) The corresponding FTLE field.}
\label{Fig:Circular}
\end{figure}

We begin with a simple toy example that intuitively demonstrates a structure in the field. The field is designed with a small ring-shaped flow on a 2D plane, centered at the origin. At the boundary of the ring, a notable difference in velocity creates a distinct structure. The velocity field is given by:
$$\frac{d\vx}{dt}=\begin{cases}
    (y,-x)^{T}\quad&(1-\varepsilon)^{2}\leq(x+y)^{2}\leq(1+\varepsilon)^{2}\\
    (0,0)^{T}\quad&\text{otherwise}
\end{cases}$$
where $\varepsilon$ is some small constant, in our experiment it is taken to be $0.1$, we run the experiment with domain $[-2,2]\times [-2,2]$ and $\delta={4}/{399}$ and time range $[0,1]$.

As we examine the visualizations in Figure \ref{Fig:Circular} (a-d), we can observe that all the visualizations successfully highlight the structure we expected, which corresponds to the boundaries of the ring. It is important to note that in Figure \ref{Fig:Circular} (a-c), there is a large blank area that may initially appear as a visualization issue. However, this blank area corresponds to regions where the field remains static. Consequently, the LTVE at those points is undefined because the exact LTV is zero. The same theoretical result holds for the FTLE case, where the LTVE is expected to be undefined in static regions. However, in practice, the computation of numerical gradients involved in the FTLE calculation may introduce computational errors, preventing exact zero values from being obtained. This discrepancy explains the visual differences between the LTVE and FTLE visualizations. It is worth noting that this issue only arises when the field is static and does not affect the visualization of LCS.

\subsection{Standing Wave}
\label{SubSec:StandingWave}

\begin{figure}[h!]
\centering
(a)\includegraphics[trim=40 0 40 0, clip, width=0.45\linewidth]{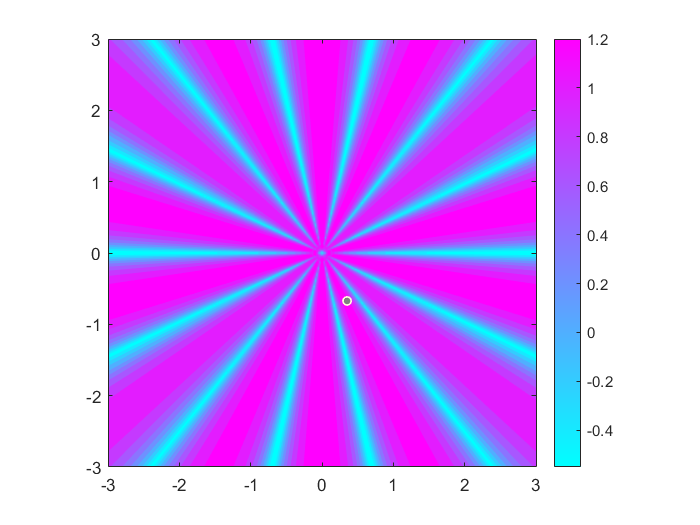}
(b)\includegraphics[trim=40 0 40 0, clip, width=0.45\linewidth]{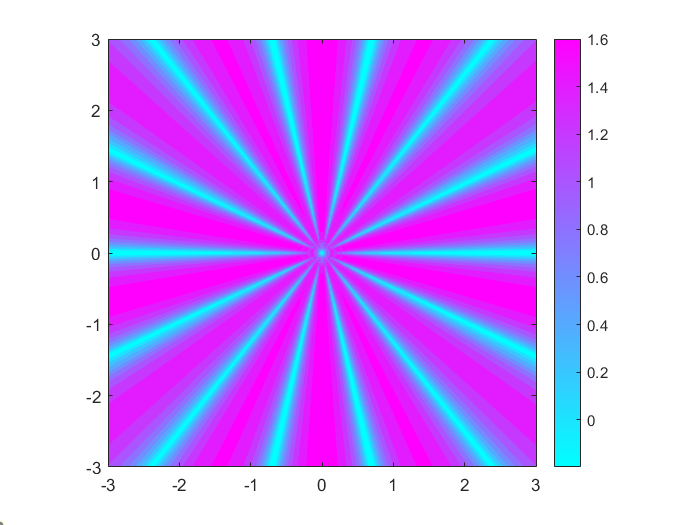}
(c)\includegraphics[trim=40 0 40 0, clip, width=0.45\linewidth]{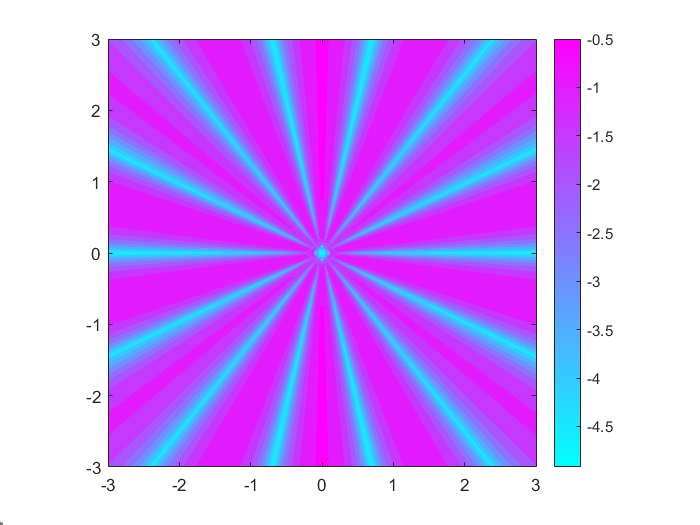}
(d)\includegraphics[trim=40 0 40 0, clip, width=0.45\linewidth]{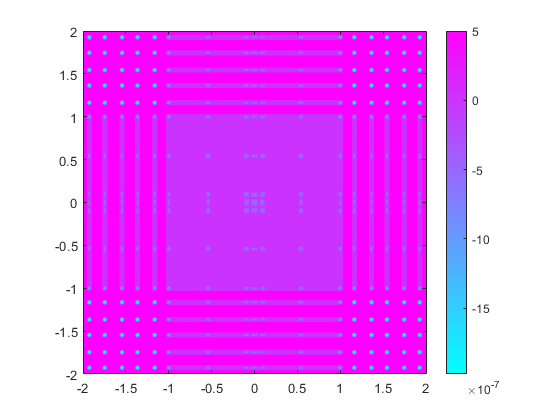} 
\caption{(Section \ref{SubSec:StandingWave}) The computed LTVE using (a) the normalized 2-norm, (b) the Fr\'{e}chet distance, and (c) the Hausdorff metric. (d) The corresponding FTLE field.} 
\label{Fig:StandingWave}
\end{figure}

In this example, we consider a simple flow where we can analytically determine the trajectory of each individual particle. The FTLE only considers the final location of the particles and disregards any information from the intermediate steps. As a result, the FTLE field heavily depends on the chosen final visualization time. To demonstrate the superiority of our proposed LTVE, we introduce a velocity field given by:
\[
\frac{d\mathbf{x}}{dt} = \left(\frac{\omega\varepsilon\cos(\omega t)}{1+\varepsilon\sin(\omega t)}\right) \mathbf{x}
\]
Here, $\varepsilon = \varepsilon(\mathbf{x}_0)$, where $\mathbf{x}_0$ represents the initial location of the particle at $t=0$. We model $\varepsilon$ as $\varepsilon = \varepsilon_{\max}\cos(k\theta)$, where $\theta = \tan^{-1}\left(\frac{y}{x}\right)$ is the polar angle representation of $\mathbf{x}_0$, and $\varepsilon_{\max} = 0.8$ controls the magnitude of the oscillations. The trajectory of each particle is given by $\mathbf{x}(t) = (1+\varepsilon\sin(\omega t))\mathbf{x}_0$. By considering this specific flow, we can demonstrate that our proposed LTVE provides a better understanding of the underlying flow compared to the FTLE. Unlike the FTLE, which only considers the final location of the particles, the LTVE takes into account the information from all intermediate steps. This allows for a more comprehensive analysis of the flow dynamics.

This specific flow exhibits two important properties. First, for $\theta = \frac{(2n-1)\pi}{2k}$, where $n$ and $k$ are integers, we observe that $\varepsilon = 0$, resulting in stationary particles for all time. Second, for all other particles, we observe oscillatory behavior with a period of $2\pi/\omega$. These particles return to their original locations after this period, leading to an FTLE value of 0 throughout the computational domain. In this numerical example, we choose $\omega = 8\pi$ and $k = 7$. We compute the trajectories for a time period of $[0,1]$. The computational domain is set to $[-3,3]^2$, and we use a grid spacing of $\delta = {6}/{299}$ in the definition of the LTVE. Figure \ref{Fig:StandingWave}(d) shows the FTLE field, which evaluates to 0 with some noise on the order of $10^{-7}$. However, in Figures \ref{Fig:StandingWave}(a-c), the structure of the flow is clearly visualized, with LTVE values ranging from $10^{-1}$ to $10^{-2}$. This significant difference in magnitudes demonstrates the distinguishability of our proposed quantity from the classical FTLE result.

\subsection{Evaluation with 3-Trajectories}
\label{SubSec:3trajEval}

\begin{figure}[h!]
    \centering
    (a)\includegraphics[trim=40 0 40 0, clip, width=0.45\linewidth]{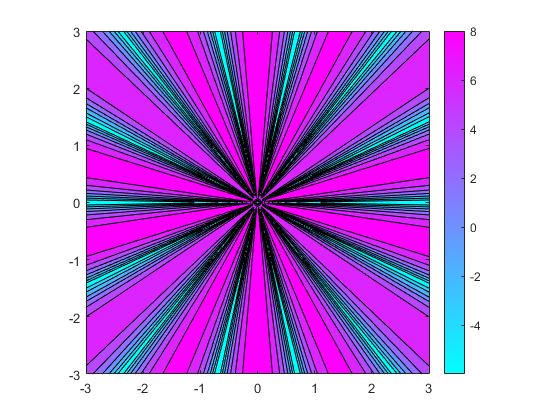}
    (b)\includegraphics[trim=40 0 40 0, clip, width=0.45\linewidth]{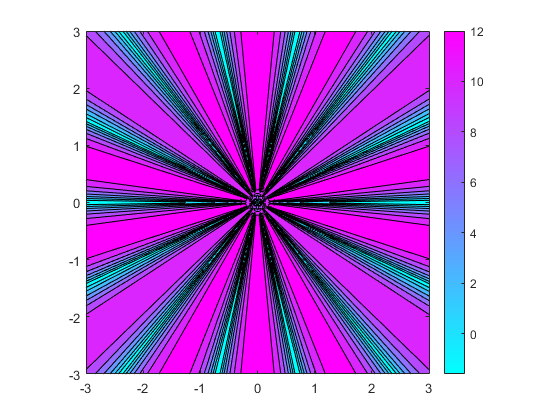}
    (c)\includegraphics[trim=40 0 40 0, clip, width=0.45\linewidth]{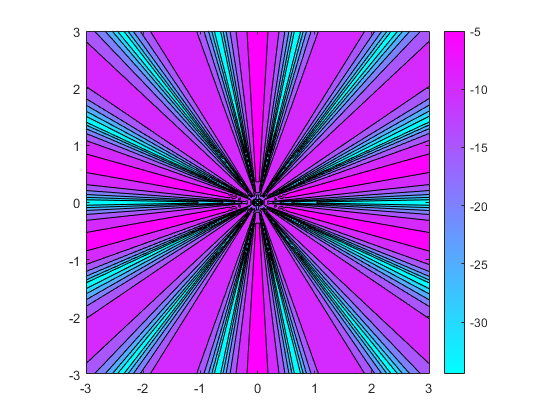}
    (d)\includegraphics[trim=40 0 40 0, clip, width=0.45\linewidth]{StandingWaveFTLE.png} 
    \caption{(Section \ref{SubSec:3trajEval}) The computed LTVE using (a) the normalized 2-norm, (b) the Fr\'{e}chet distance, and (c) the Hausdorff metric. (d) The corresponding FTLE field.} 
    \label{Fig:StandingWave3traj}
\end{figure}

In this section, we recompute the LTVE for the Standing Wave example introduced in Section \ref{SubSec:StandingWave} with a different setting. Specifically, we use only 3 sample points for the trajectories and choose a final time of $T=0.125$, while keeping all other parameters unchanged.

As shown in Figure \ref{Fig:StandingWave3traj}, the results obtained from evaluating the LTVE with only three sampling points in the trajectories align with the findings reported in Section \ref{SubSec:StandingWave}. The LTVE values obtained exhibit a substantial difference, spanning several orders of magnitude, compared to the FTLE values. This stark contrast demonstrates that even with the addition of just one extra sampling point to the trajectories, our LTVE quantity is capable of identifying structures that the FTLE fails to capture.

\subsection{Double Gyre Flow}
\label{SubSec:DoubleGyre}

\begin{figure}[h!]
\centering
(a)\includegraphics[trim=30 80 30 80, clip, width=0.45\linewidth]{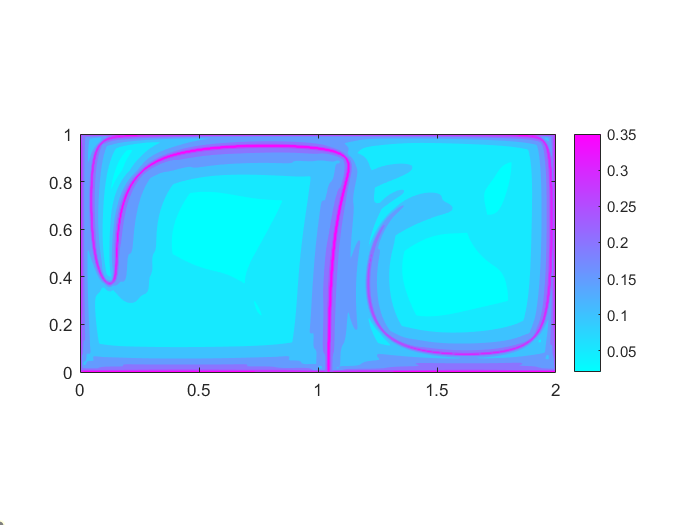}
(b)\includegraphics[trim=30 80 30 80, clip, width=0.45\linewidth]{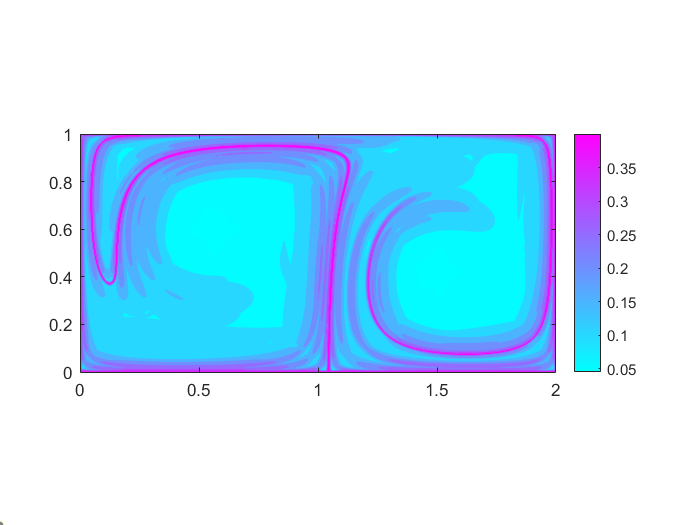} \\
(c)\includegraphics[trim=30 80 30 80, clip, width=0.45\linewidth]{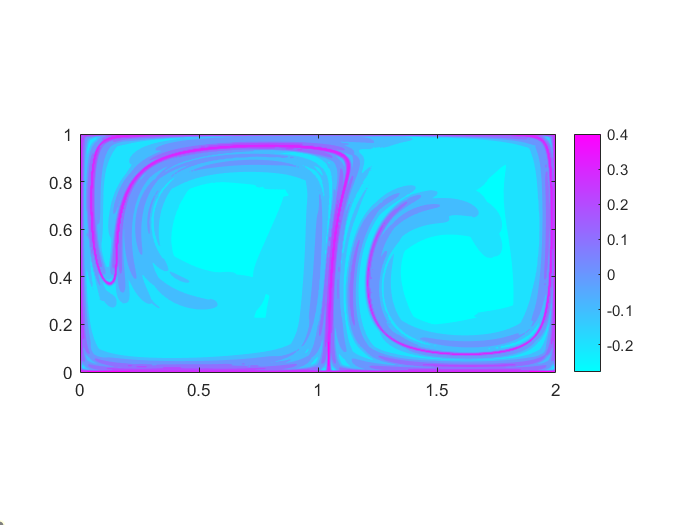}
(d)\includegraphics[trim=30 80 30 80, clip, width=0.45\linewidth]{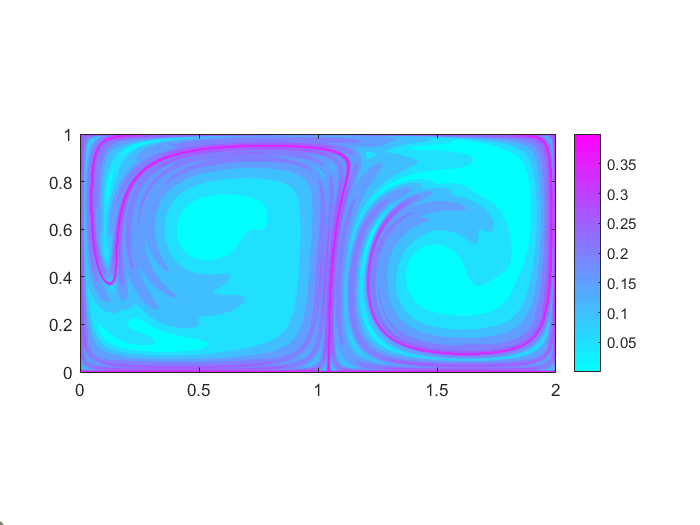}
\caption{(Section \ref{SubSec:DoubleGyre}) The computed LTVE using (a) the normalized 2-norm, (b) the Fr\'{e}chet distance, and (c) the Hausdorff metric. (d) The corresponding FTLE field.}
\label{Fig:Double}
\end{figure}

\begin{figure}[h!]
\centering
(a)\includegraphics[trim=30 60 30 50, clip, width=0.45\linewidth]{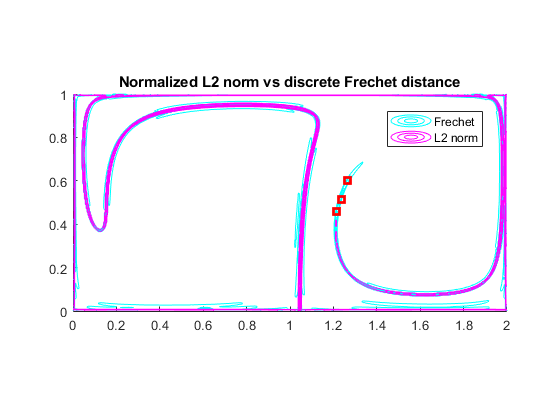}
(b)\includegraphics[trim=30 60 30 50, clip, width=0.45\linewidth]{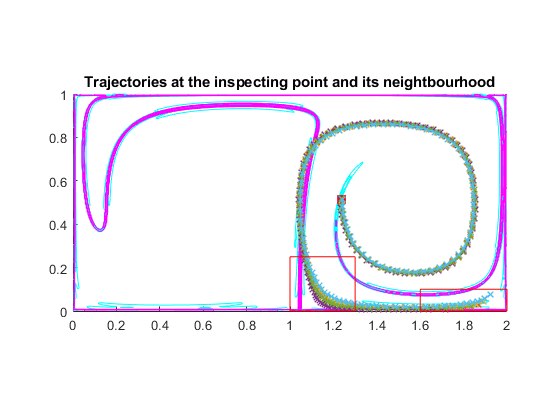}
(c)\includegraphics[trim=20 0 20 0, clip, width=0.45\linewidth]{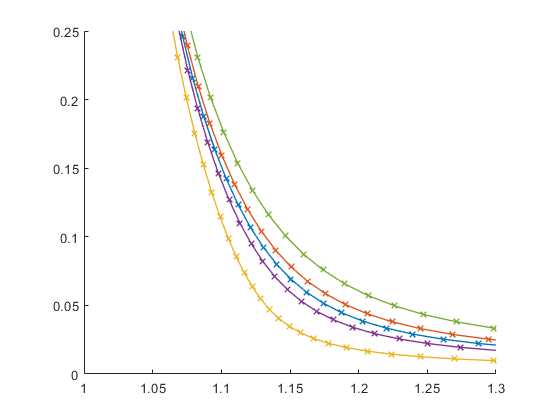}
(d)\includegraphics[trim=30 90 30 80, clip, width=0.45\linewidth]{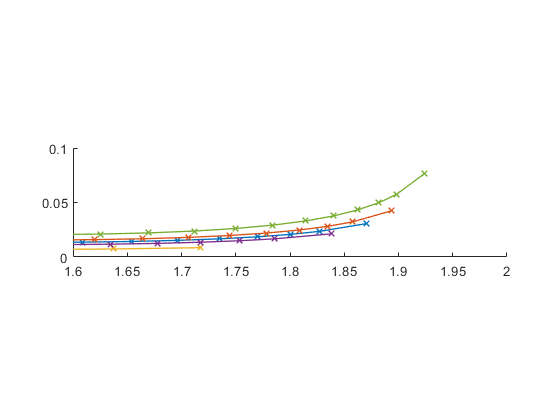}
\caption{(Section \ref{SubSec:DoubleGyre}) The computed LTV using (a) the normalized 2-norm and the Fr\'{e}chet distance. (b) We also plot the trajectory of several particles. (c) The zoom-in to the red boxed region on left and (d) the red boxed region on right in (b).}
\label{fig:DoubleTrajectory}
\end{figure}

\begin{figure}[h!]
\centering
\includegraphics[clip, width=0.45\linewidth]{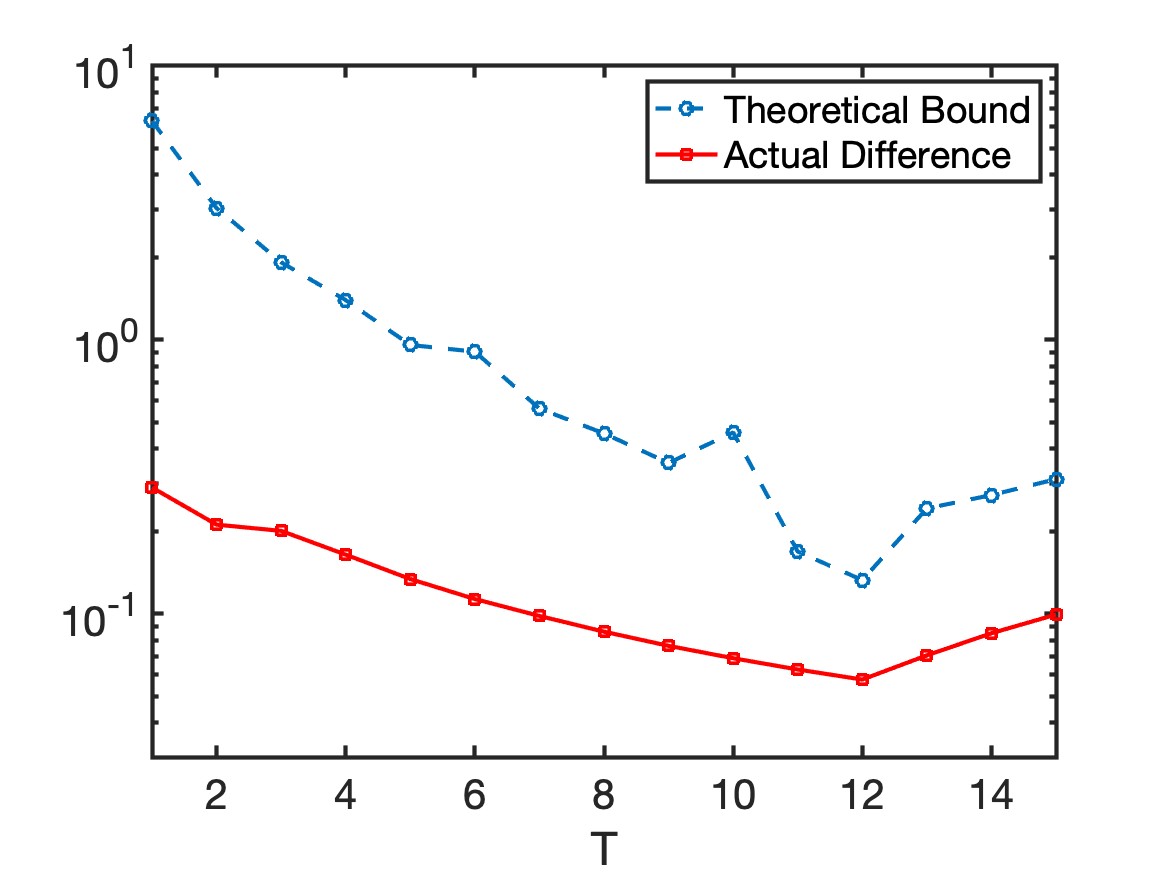}
\caption{(Section \ref{SubSec:DoubleGyre}) The actual difference between the LTVE and the FTLE at different time. The estimate from equation (\ref{Eqn:LTVEvsFTLE}) is shown using the dashed line.}
\label{Fig:LTVEvsFTLE} 
\end{figure}

In this section, we proceed with a common example used for computing the LCS of dynamic systems, known as the double gyre flow. We use this example to illustrate how our quantity is able to accurately capture the LCS. As a convention, we choose the domain to be $[0,2]\times[0,1]$ with a spacing of $\delta={1}/{249}$. The time range is taken to be $[0,15]$ with 100 sampling points along the trajectory. The velocity field is given by:
$$
\frac{d\mathbf{x}}{dt}=\begin{pmatrix}
    -\pi A\sin(\pi(ax^{2}+bx))\cos(\pi y)\\
    \pi A(2ax+b)\cos(\pi(ax^{2}+bx))\sin(\pi y)
\end{pmatrix}
$$
where $A=0.1$, $a=0.1\sin\left({\pi t}/{5}\right)$, and $b=1-2a$. Below, we provide a visualization of our quantity compared with FTLE. We then analyze the trajectory in the part that demonstrates differences between the two metrics in the visualization. We also provide the results of an actual runtime experiment, where we vary the mesh size and trajectory length. Finally, we compare the effects of using different relaxation schemes.

The results of the computation are visualized in Figures \ref{Fig:Double}(a-c). Comparing these results to the FTLE field plotted in Figure \ref{Fig:Double}(d), we can see that our proposed quantity forms a sharp and clean ridge that agrees with the FTLE ridge. However, there is a slight difference in the tail of the structure between the results obtained from the three metrics. To provide a better visualization, we plotted the LTVE values computed under different metrics and overlapped them in Figure \ref{fig:DoubleTrajectory}(a). Specifically, we selected a point located at the tail and plotted the trajectories at that point and in its neighborhood for comparison, as shown in Figure \ref{fig:DoubleTrajectory}(b). Upon closer inspection, we can observe that the trajectory at the particular point does not exhibit a strong expelling or attracting feature. However, there are slight variations in the trajectories, as seen in the region boxed in red and visualized in Figures \ref{fig:DoubleTrajectory}(c-d). We notice that although the result evaluated from the Fr\'{e}chet distance appears to be more consistent with the tail part of the FTLE, Figures \ref{fig:DoubleTrajectory}(c-d) show that the perturbation of the trajectories is not as significant as expected for being a LCS. This indicates that even along the FTLE ridge, the effect of perturbing the trajectories can still vary. The structure captured at a certain point may be much weaker, and our proposed quantity can capture such features by using a different metric.

To validate the relationship between the LTVE and FTLE, we computed the actual differences at different times and compared them with the estimations derived from equation (\ref{Eqn:LTVEvsFTLE}). The resulting differences are illustrated in Figure \ref{Fig:LTVEvsFTLE}. In the figure, each blue point represents the theoretical bound estimated from equation (\ref{Eqn:LTVEvsFTLE}), while each red point represents the actual difference evaluated when fixing the time span to $[0,T]$. We observe that the red points, representing the actual differences, consistently fall below the blue points, which correspond to the theoretical estimations. This alignment between the experimental results and our derived equation (\ref{Eqn:LTVEvsFTLE}) confirms the validity of our theoretical derivation.

\begin{figure}[h!]
    \centering
    (a)\includegraphics[trim=30 80 50 80, clip, width=0.45\linewidth]{LTVEDoubleGyreR2k.png}
    (b)\includegraphics[trim=30 80 50 80, clip, width=0.45\linewidth]{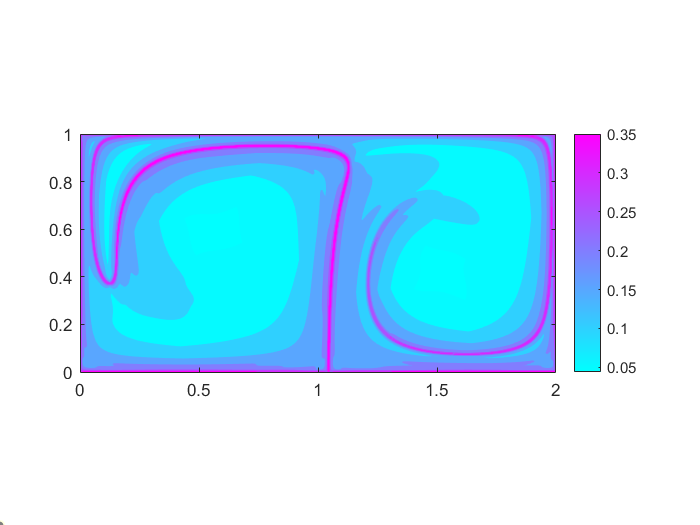}
    (c)\includegraphics[trim=30 80 50 80, clip, width=0.45\linewidth]{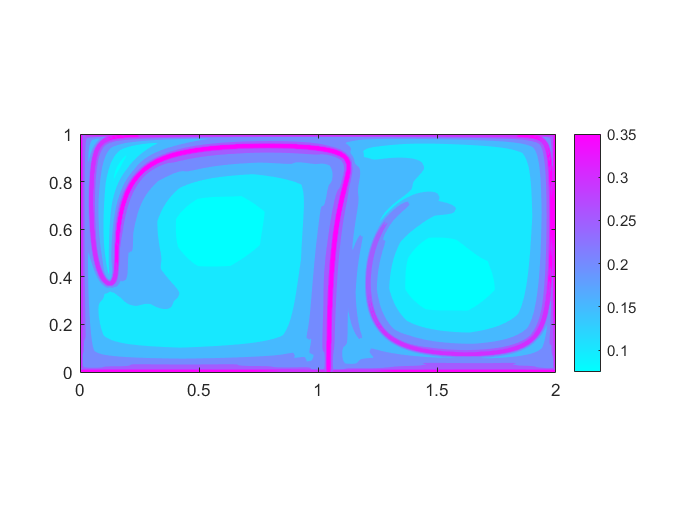}
    (d)\includegraphics[trim=30 80 50 80, clip, width=0.45\linewidth]{DoubleGyreFTLE.png}
    \caption{(Section \ref{SubSec:Relax}) The computed LTVE using the normalized 2-norm with (a) first order relaxation scheme (b) second order relaxation scheme and (c) third order relaxation scheme, (d) is the FTLE field for comparison.}
    \label{fig:Relax}
\end{figure}

\subsection{Relaxation Schemes}
\label{SubSec:Relax}

Here, we present a comparative study on the effect of different relaxation schemes on the visualization results of the double gyre flow. From Figure \ref{fig:Relax}, we can observe that there is not a significant difference between (a) and (b). However, in (c), our quantity highlights a thicker line of structure. Furthermore, the structure that we focused on in Figure \ref{fig:DoubleTrajectory}(a) is now consistent with the result obtained using the Fr\'{e}chet distance. This result aligns with our claim in Section \ref{SubSec:DoubleGyre} that the structure appearing at the tail is weaker and more sensitive to the initial condition. Therefore, by adopting a higher-order relaxation scheme, it is more likely to capture the structure that only becomes significant along certain directions.

\begin{table}[h!]
  \centering
  (a)
  \begin{tabular}{|c|c|c|c|}
    \cline{2-4}
    \multicolumn{1}{c|}{}&\multicolumn{3}{c|}{Metric}\\
    \hline
    mesh size $\delta$&Normalized $L^{2}$-norm&Fr\'{e}chet distance&Hausdorff metric\\
    \hline
    ${1}/{49}$&$0.209$&$21$&$47$\\
    \hline
    ${1}/{99}$&$0.835$&$85$&$193$\\
    \hline
    ${1}/{149}$&$1.878$&$192$&$436$\\
    \hline
    ${1}/{199}$&$3.334$&$342$&$776$\\
    \hline
    ${1}/{249}$&$5.225$&$535$&$1231$\\\hline
  \end{tabular} \\ \vspace{0.5cm}
  (b)
     \begin{tabular}{|c|c|c|c|}
        \cline{2-4}
        \multicolumn{1}{c|}{}&\multicolumn{3}{c|}{Metric}\\
        \hline
        trajectory length $k$&Normalized $L^{2}$-norm&Fr\'{e}chet distance&Hausdorff metric\\
        \hline
        $5$&$0.175$&$0.355$&$0.461$\\\hline
        $10$&$0.196$&$0.912$&$1.392$\\\hline
        $15$&$0.226$&$1.831$&$2.857$\\\hline
        $20$&$0.250$&$3.051$&$4.893$\\\hline
        $25$&$0.277$&$4.565$&$7.415$\\\hline
        $30$&$0.304$&$6.454$&$10.506$\\\hline
        $35$&$0.330$&$8.578$&$14.166$\\\hline
        $40$&$0.370$&$11.459$&$18.520$\\\hline
        $45$&$0.383$&$14.208$&$23.864$\\\hline
        $50$&$0.427$&$17.457$&$28.483$\\\hline
        $55$&$0.463$&$21.980$&$34.331$\\\hline
        $60$&$0.463$&$24.122$&$41.441$\\\hline
    \end{tabular}
  \caption{(Section \ref{SubSec:RunTime}) (a) The CPU run time in seconds for different mesh sizes over different metrics. (b) The CPU run time in seconds for different lengths of trajectory over different metrics.}
  \label{tab:TimeVsMeshLength}
\end{table}

\begin{figure}[h!]
\centering
(a)\includegraphics[clip, width=0.45\linewidth]{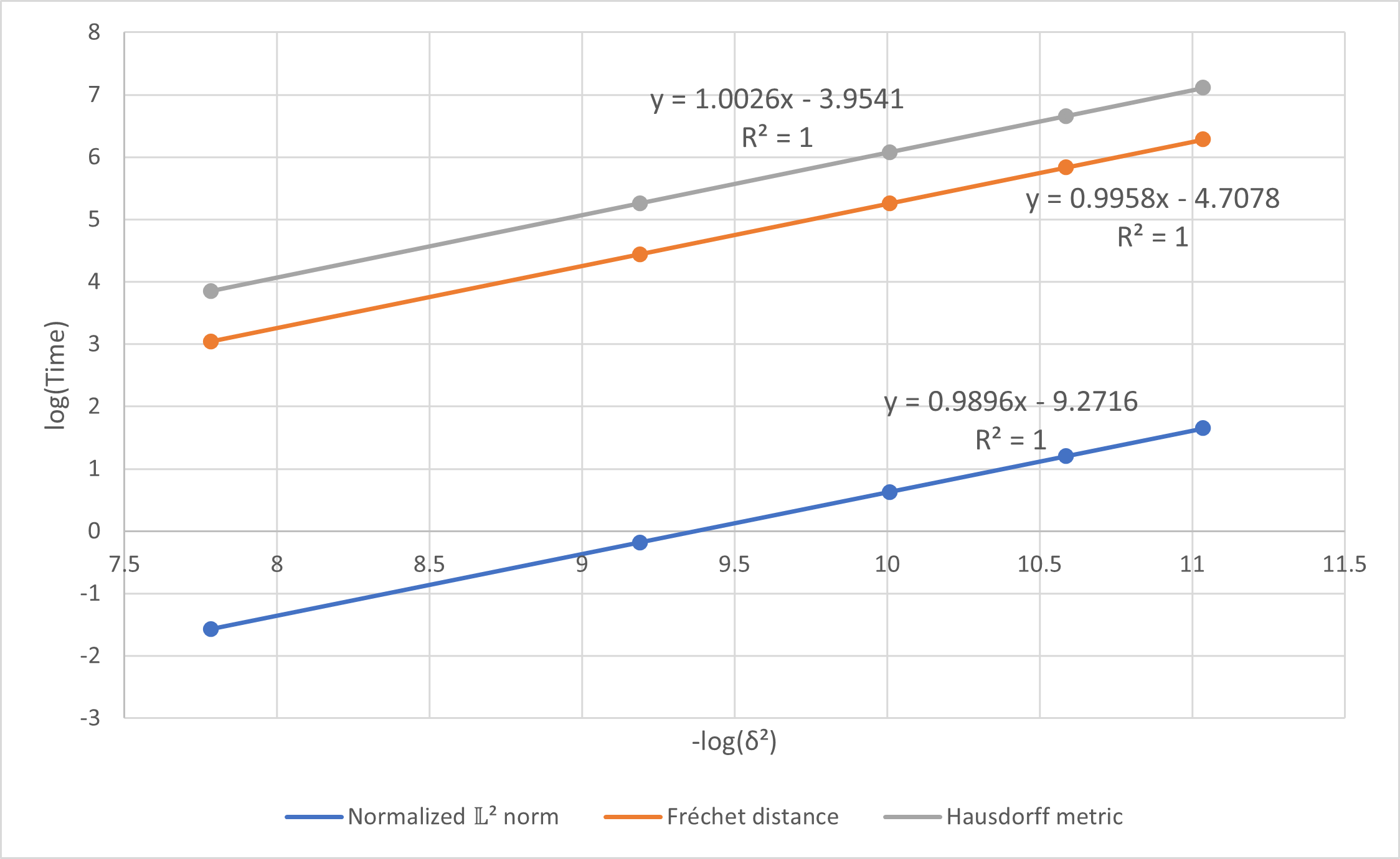}
(b)\includegraphics[clip, width=0.45\linewidth]{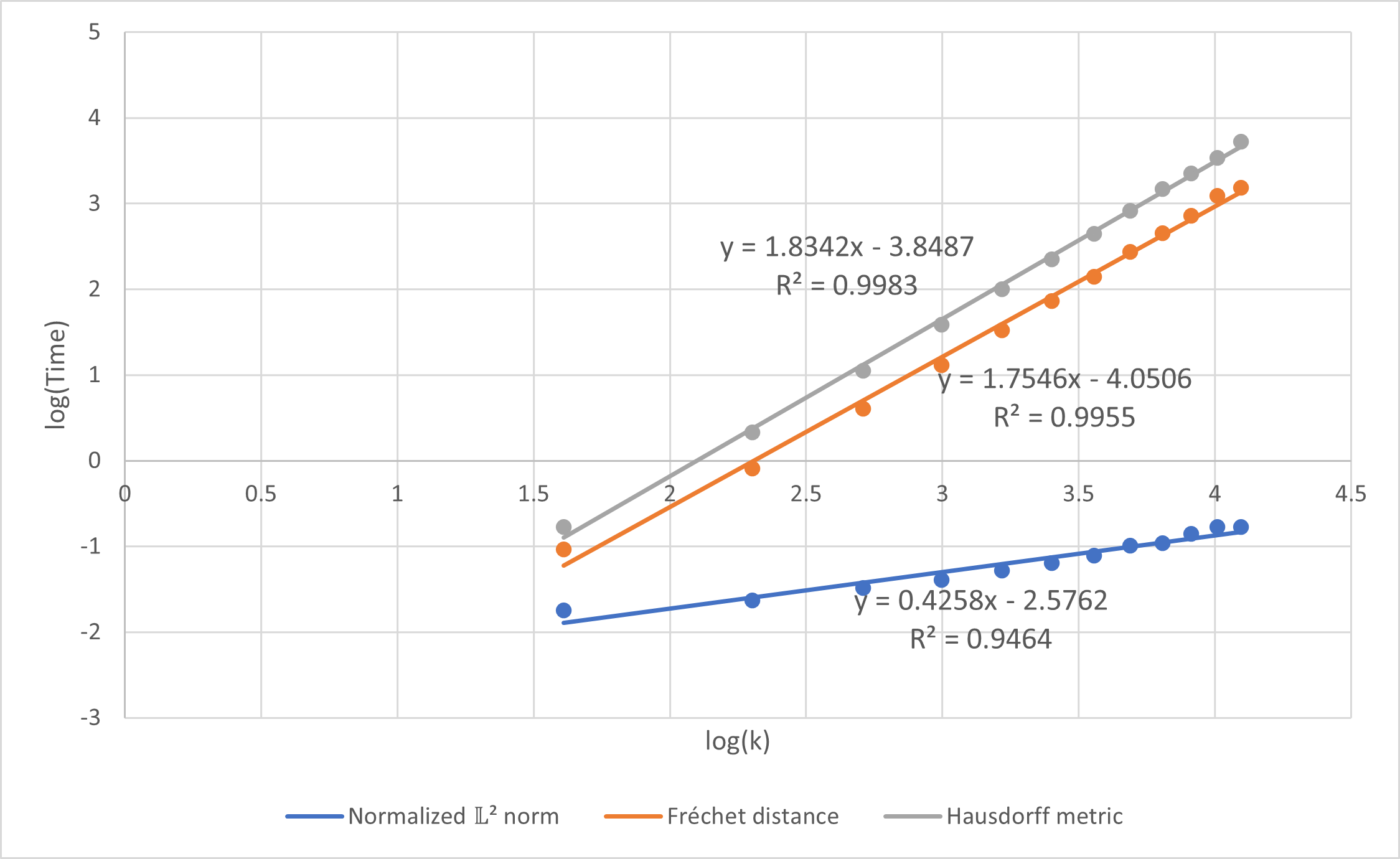}
\caption{(Section \ref{SubSec:RunTime}) (a) The log time versus $\log \delta^{-2}$, and (b) The log time versus the log of the trajectory length.}
\label{Fig:RunTime} 
\end{figure}

\subsection{Runtime Performance}
\label{SubSec:RunTime}
In this section, we provide the actual CPU run time results while varying the trajectory length and mesh size. All the results were obtained using a laptop computer with a 4-core 11th Gen Intel i5 CPU and 16GB of RAM. It is important to note that our code is not fully optimized, and the main purpose of reporting these runtimes is to showcase the complexity under different metrics. To conduct our measurements, we utilized the double gyre flow from Section \ref{SubSec:DoubleGyre} as the underlying field. This choice was motivated by its complexity, which allows us to present a comprehensive example for visualizing LCS.

\paragraph{Mesh Size.} 
First, we conducted performance measurements using five selected mesh sizes, denoted as $\delta$, while keeping other parameters the same as in our previous experiments. The recorded times are presented in Table \ref{tab:TimeVsMeshLength}(a), and we plotted the logarithm of time against $\log {\delta^{-2}}$ in Figure \ref{Fig:RunTime}(a). Notably, the figure exhibits a clear linear relationship with a slope of one, which aligns with our complexity analysis. Furthermore, based on our complexity analysis, the $y$-intercept of the plot should depend on the domain size and the complexity of the metric used to compare trajectories. In this case, since the domain size is fixed, the slope is primarily determined by the metric. This observation is supported by the data, as the plot of the normalized $L^{2}$ norm displays a significantly smaller $y$-intercept compared to the other two metrics, indicating its linear complexity. Regarding the Fr\'{e}chet distance and Hausdorff metric, the computation of the Fr\'{e}chet distance involves dynamic programming, resulting in a relatively lower actual runtime compared to the Hausdorff metric, which requires comparing all pairs of trajectory points to compute the maximum distance. This explains why the Fr\'{e}chet distance has a lower runtime compared to the Hausdorff metric.

\paragraph{Trajectory Length.} 
We then measured the performance for five different trajectory lengths, keeping the other parameters constant. The results are presented in Table \ref{tab:TimeVsMeshLength}(b). Once again, the normalized $L^{2}$-norm outperformed the other two metrics due to its linear complexity. The Fr\'{e}chet distance metric performed slightly better than the Hausdorff metric, consistent with our previous analysis. Our complexity analysis suggests that the total runtime should be linear with the complexity of the metric, which is confirmed by the plot in Figure \ref{Fig:RunTime}(b). The slope of the plot is upper bounded by the actual complexity of the metric. The slight difference in runtime between the metrics may be due to hidden compiler optimizations that are hard to control. Nevertheless, the complexity analysis provides a valid upper bound for the runtime.

\subsection{Field with Artificial FTLE Ridge}
\label{SubSec:ArtificialRidge}

\begin{figure}[h!]
\centering
(a)\includegraphics[trim=60 20 50 0, clip, width=0.45\linewidth]{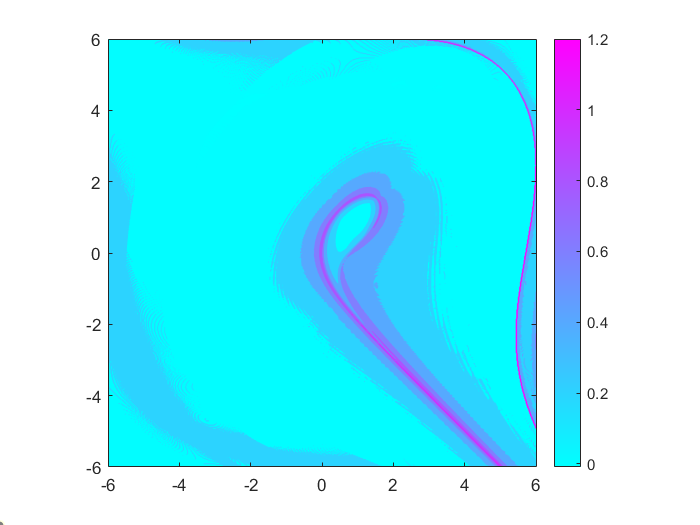}
(b)\includegraphics[trim=60 20 50 0, clip, width=0.45\linewidth]{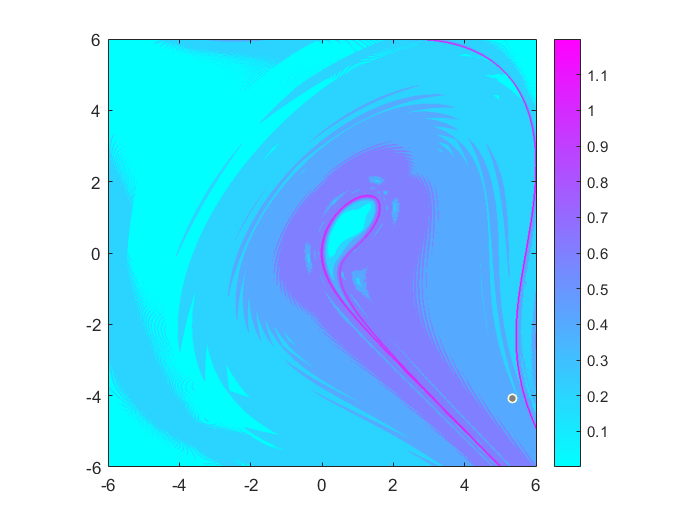}
(c)\includegraphics[trim=60 20 50 0, clip, width=0.45\linewidth]{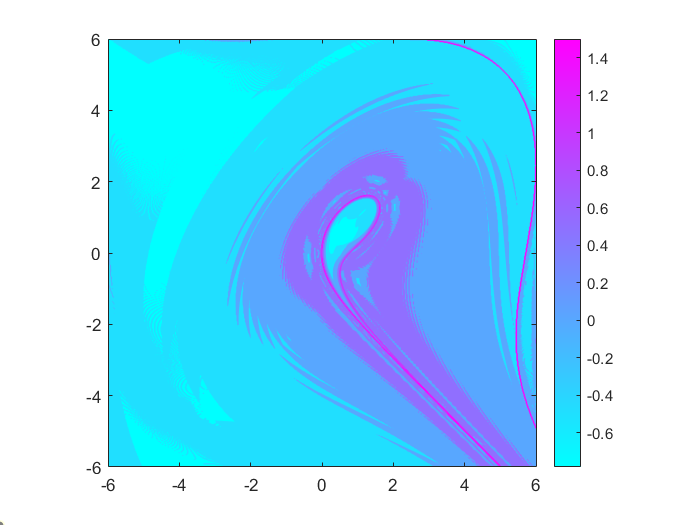}
(d)\includegraphics[trim=60 20 50 0, clip, width=0.45\linewidth]{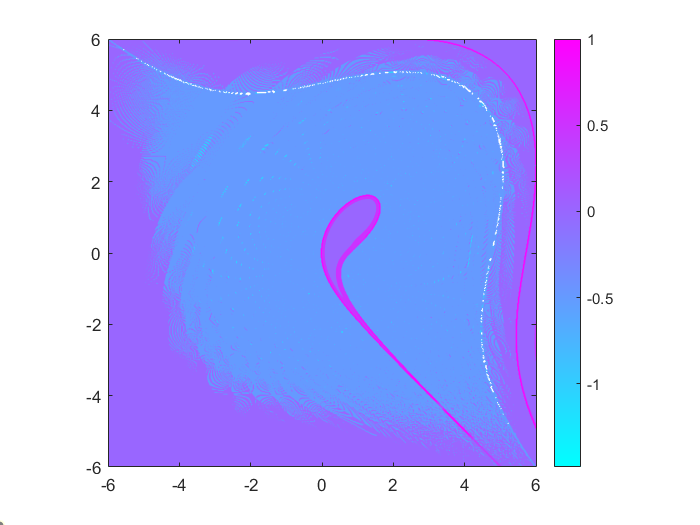}
\caption{(Section \ref{SubSec:ArtificialRidge}) The computed LTVE using (a) the normalized 2-norm, (b) the Fr\'{e}chet distance, and (c) the Hausdorff metric. (d) The corresponding FTLE field. }
\label{Fig:ArtificialLCS}
\end{figure}

In \cite{leu11}, an Eulerian approach was proposed for computing the FTLE field, which highlighted the issue of artificial ridges arising from the boundary condition upon solving the trajectories. To illustrate this issue, a specific example was presented, where the velocity field is given by:
$$\frac{d\vx}{dt}=\begin{pmatrix}
x-y^{2}\
-y+x^{2}
\end{pmatrix}$$
For this example, the computational domain chosen was $[-6,6]\times [-6,6]$ with a spatial resolution of $\delta={12}/{599}$. The time range was set to $[0,5]$, with 100 sampling points for each trajectory.

As depicted in Figure \ref{Fig:ArtificialLCS}(a-c), our proposed method yields results that closely resemble those obtained using the FTLE field. However, it is worth noting that our method generally provides a cleaner visualization due to the sharp values at the structural points. It is important to acknowledge that this example contains an artificial LCS that does not exist in its natural definition. To address this issue, various techniques can be employed, such as the velocity extension algorithm proposed in \cite{doi:10.1063/1.3276061} or exploring an Eulerian approach for our quantity as demonstrated in \cite{leu11}, which may help mitigate the presence of artificial ridges.


\subsection{Ocean Surface Current Analyses Real-time (OSCAR)}
\label{SubSec:OSCAR}

\begin{figure}[h!]
    \centering
    (a)\includegraphics[trim=30 80 30 80, clip, width=0.45\linewidth]{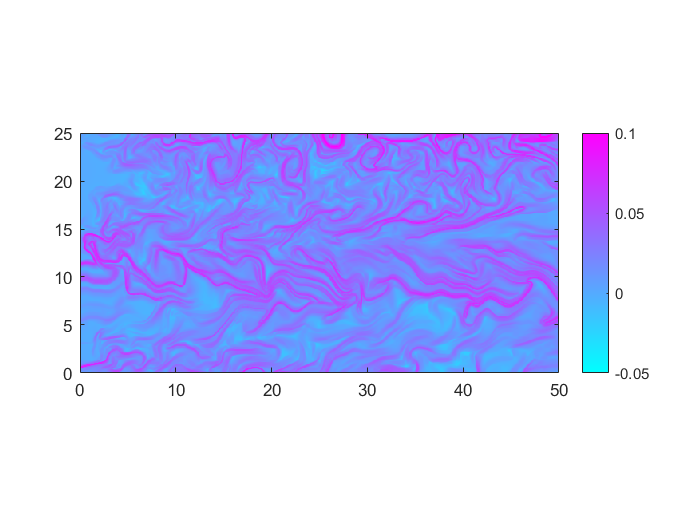}
    (b)\includegraphics[trim=30 80 30 80, clip, width=0.45\linewidth]{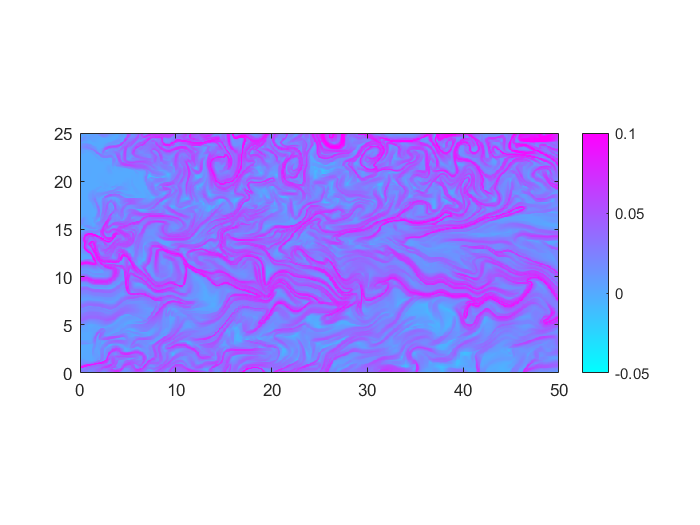} \\ 
    (c)\includegraphics[trim=30 80 30 80, clip, width=0.45\linewidth]{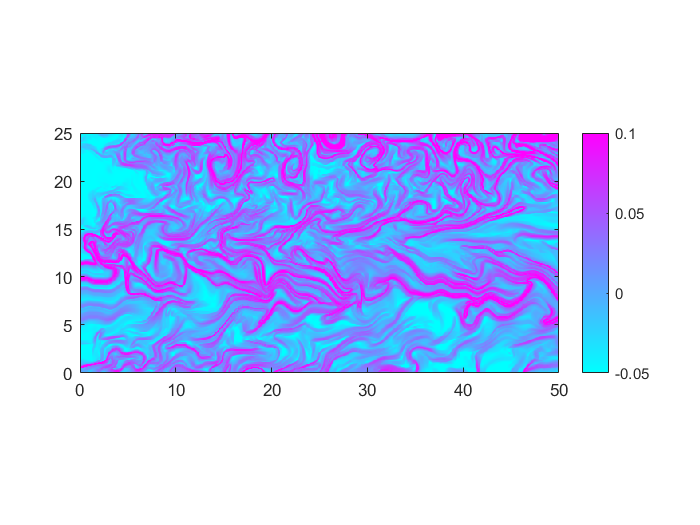}
    (d)\includegraphics[trim=30 80 30 80, clip, width=0.45\linewidth]{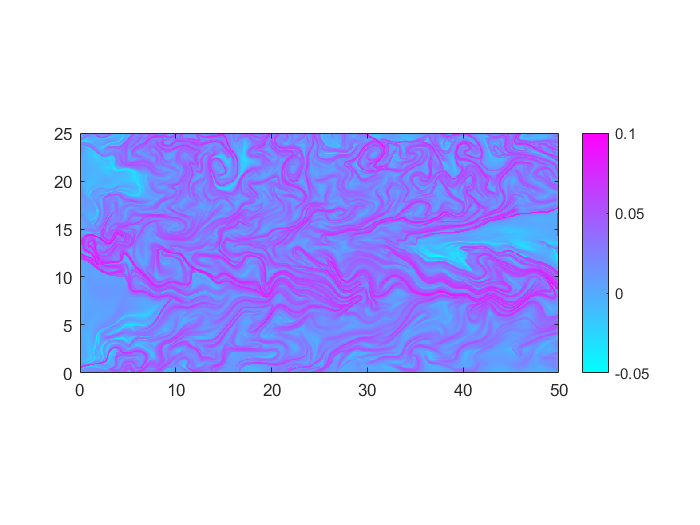}
    \caption{(Section \ref{SubSec:OSCAR}) The computed LTVE using (a) the normalized 2-norm, (b) the Fr\'{e}chet distance, and (c) the Hausdorff metric. (d) The corresponding FTLE field \cite{youwonleu17,youleu20}. }
    \label{Fig:OSCAR}
\end{figure}

To assess the efficacy and reliability of our proposed approach, we conducted experiments on a real-life dataset sourced from the Ocean Surface Current Analyses Real-time (OSCAR)\footnote{\url{http://www.esr.org/oscar_index.html}}. The OSCAR dataset captures ocean surface flow in a wide region spanning from $-80^{\circ}$ to $80^{\circ}$ latitude and $0^{\circ}$ to $360^{\circ}$ longitude. We acquired the OSCAR data from the JPL Physical Oceanography DAAC developed by ESR. In this particular numerical example, our focus centers on the vicinity of the Line Islands, encompassing the latitude range of $-17^{\circ}$ to $8^{\circ}$ and the longitude range of $180^{\circ}$ to $230^{\circ}$. The temporal resolution of the OSCAR data is approximately 5 days, and the spatial resolution is $1/3^{\circ}$ in each direction.

To enhance the visual representation of our solution, we performed velocity data interpolation, yielding a finer resolution of $0.25$ days in the temporal dimension and $1/12^{\circ}$ in each spatial dimension. Subsequently, we examined the ocean surface current during the first 50 days of the year 2014 to derive the forward FTLE field $\sigma_0^{50}(\mathbf{x})$. The computational results obtained from our proposed algorithm, utilizing various metrics, are presented in Figures \ref{Fig:OSCAR}(a-c). Notably, these results agree with the FTLE field computed in \cite{youwonleu17,youleu20}, as depicted in Figure \ref{Fig:OSCAR}(d), in general, but show differences in various regions. This provides insight into the flow structures that are worth further investigation.


\section{Conclusions}
\label{Sec:Conclusion}

In this study, we propose a novel quantity based on the trajectory metric space, which bears resemblance to the FTLE field. We evaluate the efficacy of this new quantity with several examples, including one that reveals significant differences between our quantity and the FTLE field. This example serves as a motivation for the use of trajectory analysis in LCS visualization. We further test our quantity with other examples that demonstrate its ability to handle complex flows that arise in natural and certain flow conditions. In all cases, we observe that our quantity captures the general structure that the FTLE field captures, albeit with slight differences contingent upon the specific metric employed. We find that the differences between our quantity and the FTLE field are minor in the field we examined, but there may be fields where significant differences exist. However, the direct application of the Fr\'{e}chet distance or the Hausdorff metric is prohibitively expensive, and the time cost does not justify the marginal improvement in results.

In conclusion, our proposed method represents a generalization of the traditional FTLE approach with trajectories for visualizing Lagrangian coherent structures (LCS). We recommend utilizing the normalized $L^{2}$-norm as the primary metric, as it generally provides efficient and satisfactory results without significant differences compared to other metrics. Unless a specific application case presents a compelling reason to justify the use of a different metric, the normalized $L^{2}$-norm is sufficient for capturing the desired flow structures. To the best of our knowledge, no such application case has been identified thus far. However, we plan to further investigate this aspect in future studies to explore potential scenarios where alternative metrics may offer distinct advantages. By doing so, we aim to enhance our understanding of the implications and potential benefits of different metrics for visualizing LCS. 

Although we did not present any experiments of computing our quantity on higher dimensional domains in this paper, we believe that the framework for our quantity is general and not specifically dependent on the dimension, making it applicable to fields with higher dimensions.

In future research, we plan to extend our framework to dynamic cases and explore scenarios where a different metric could lead to a significantly different and useful result. Such results would be intriguing because they would also imply that the classical FTLE approach fails in these cases as it is a special case of applying the normalized $L^{2}$ norm. Additionally, we plan to investigate algorithms that can accelerate metric calculations since we found that metric complexity significantly affects practical runtime. To achieve dynamic cases, we need to identify a way of updating trajectories efficiently, enabling online updates of the quantity in a practical way.

\section*{Acknowledgments}
Leung is partially supported by the Hong Kong RGC (grant no. 16302223).

\bibliographystyle{plain}
\bibliography{bibtex,syleung}

\end{document}